\documentclass[11pt]{article}
    \usepackage{url}
    \usepackage{verbatim}
    \usepackage[titletoc]{appendix}
    \usepackage{graphicx}
	\usepackage{amsmath}
	\usepackage{amsthm}
	\usepackage{amsfonts}
	\usepackage{graphicx}
    \usepackage{nicefrac}
    \usepackage{longtable}
    \usepackage{color}
    \usepackage{graphicx, amssymb, subcaption,graphics}

    \newcommand{\norm}[2]{\left\| #1 \right\|_{#2}}
    \newcommand{\hf}{\frac{1}{2}}
    \newcommand{\be}{\begin{equation}}
    \newcommand{\ee}{\end{equation}}

    \newcommand{\nabh}{\nabla_{\! h}}

    \newcommand{\nrm}[1]{\left\| #1 \right\|}

    \def\n{\mbox{\boldmath $n$}}

    \newcommand\dt {{\Delta t}}
     \def\0{\mbox{\boldmath $0$}}

    \newcommand{\cipgen}[3]{\left\langle #1 , #2 \right\rangle_{#3}}

\newcommand{\eipx}[2]{\left[ #1 , #2 \right]_{\rm x}}
\newcommand{\eipy}[2]{\left[ #1 , #2 \right]_{\rm y}}

	\newtheorem{thm}{Theorem}[section]
	\newtheorem{prop}[thm]{Proposition}
	\newtheorem{cor}[thm]{Corollary}
	
	\newtheorem{lem}[thm]{Lemma}
	\newtheorem{rem}[thm]{Remark}

        \usepackage{algorithm}
        \usepackage{algpseudocode}
	
\begin{document}
	\title{Convergence Analysis of a Preconditioned Steepest Descent Solver for the Cahn-Hilliard Equation with Logarithmic Potential}

	\author{
Amanda E. Diegel \thanks{Department of Mathematics and Statistics, Mississippi State University,  Mississippi State, MS 39762 (adiegel@math.msstate.edu)}
	\and			
Cheng Wang\thanks{Department of Mathematics, The University of Massachusetts, North Dartmouth, MA  02747 (Corresponding author: cwang1@umassd.edu)}	
	\and
Steven M. Wise\thanks{Department of Mathematics, The University of Tennessee, Knoxville, TN 37996 (swise1@utk.edu)} 
	}

	\maketitle
	\numberwithin{equation}{section}

	\begin{abstract}
	In this paper, we provide a theoretical analysis for a preconditioned steepest descent (PSD) iterative solver that improves the computational time of a finite difference numerical scheme for the Cahn-Hilliard equation with Flory-Huggins energy potential. In the numerical design, a convex splitting approach is applied to the chemical potential such that the logarithmic and the surface diffusion terms are treated implicitly while the expansive concave term is treated with an explicit update. The nonlinear and singular nature of the logarithmic energy potential makes the numerical implementation very challenging. However, the positivity-preserving property for the logarithmic arguments, unconditional energy stability, and optimal rate error estimates have been established in a recent work and it has been shown that successful solvers ensure a similar positivity-preserving property at each iteration stage. Therefore, in this work, we will show that the PSD solver ensures a positivity-preserving property at each iteration stage. The PSD solver consists of first computing a search direction (which requires solving a constant-coefficient Poisson-like equation) and then takes a one-parameter optimization step over the search direction in which the Newton iteration becomes very powerful. A theoretical analysis is applied to the PSD iteration solver and a geometric convergence rate is proved for the iteration. In particular, the strict separation property of the numerical solution, which indicates a uniform distance between the numerical solution and the singular limit values of $\pm 1$ for the phase variable, plays an essential role in the iteration convergence analysis. A few numerical results are presented to demonstrate the robustness and efficiency of the PSD solver.

\medskip
	
	
\medskip
	
\end{abstract}

\section{Introduction}

The Allen-Cahn (AC)~\cite{allen79} (non-conserved dynamics) and Cahn-Hilliard (CH)~\cite{cahn58} (conserved dynamics) equations are well known gradient flows with respect to the total free energy given by
\begin{align*}
    E(\phi)=\int_{\Omega} \left(F(\phi) +\frac{\varepsilon^2}{2}|\nabla \phi|^2\right) d {\bf x} ,
\end{align*}
where $\Omega \subset \mathbb{R}^d$ (with $d=2$ or $d=3$) is a bounded domain, $-1 < \phi < 1$ is the variable of interest often representing the concentration of material components in a two-phase system, $\varepsilon$ is a positive constant associated with the diffuse interface width separating the two phases, and $F$ is a given double-well potential. In this work, we consider the Flory-Huggins energy potential. Specifically, for any $\phi \in H^1 (\Omega)$ with a point-wise bound, i.e.~$-1 < \phi < 1$, the total free energy with Flory-Huggins energy potential is given by 
	\begin{equation}
	\label{CH energy}
E(\phi)=\int_{\Omega}\left( ( 1+ \phi) \ln (1+\phi) + (1-\phi) \ln (1-\phi) - \frac{\theta_0}{2} \phi^2 +\frac{\varepsilon^2}{2}|\nabla \phi|^2\right) d {\bf x} ,
	\end{equation}
where $\theta_0$ is an additional positive constant associated with the diffuse interface width. The Cahn-Hilliard (CH) equation is then an $H^{-1}$ (conserved) gradient flow of the energy functional (\ref{CH energy}) and is given by:	
\begin{align} 
  \partial_t\phi &= \nabla \cdot ( {\cal M} (\phi) \nabla \mu ) ,  
	\label{CH equation-0} 
\\
\mu &:= \delta_\phi E = \ln (1+\phi) - \ln (1-\phi) - \theta_0 \phi - \varepsilon^2 \Delta \phi ,
	\label{CH-mu-0} 
\end{align}
where ${\cal M} (\phi) >0$ is a mobility function. Based on the gradient structure of \eqref{CH equation-0}, the energy dissipation law is derived as 
	\begin{align}
\frac{d}{dt} E(\phi(t))= & \ -\int_{\Omega} {\cal M} (\phi) |\nabla \mu |^2 d {\bf x} \le 0  . 
	\label{energy-decay-rate-CH}
	\end{align}
	For simplicity of presentation, we assume that $\Omega= [0,1]^2$ with periodic boundary conditions but remark that the case with homogeneous Neumann boundary conditions can be analyzed with a similar  strategy. 
 
The free energy with the Flory-Huggins logarithmic potential is generally viewed to be more physically realistic than an energy represented by a polynomial expression since the former can be derived from regular or ideal solution theories~\cite{doi13}. On the other hand, the Flory Huggins energy potential posses a computational challenge since it is associated with a singularity as the phase variable approaches $-1$ or $1$. Indeed, the system~\eqref{CH equation-0} -- \eqref{CH-mu-0} is only well-defined if a point-wise positivity property is imposed, i.e.,  $0 < 1-\phi$ and $0 < 1+\phi$, so that the phase variable remains in the interval $(-1, 1)$. See the related works~\cite{abels09b, abels07, barrett99, miranville11, debussche95, elliott96b, elliott96c, Giorgini17a, Giorgini18, LiD2021a, miranville12, miranville04}, etc. 

For the CH equation with a polynomial approximation in the energy potential, a maximum norm bound could be carefully derived, with the help of a global-in-time $H^2$ analysis. However, such an $L^\infty$ bound turns out to be singularly $\varepsilon^{-1}$-dependent, since the surface diffusion estimate has to be used to balance the nonlinear effects; see the related work in~\cite{guo16}. In terms of an $\varepsilon^{-1}$-independent $L^\infty$ bound, the sharpest theoretical analysis in this area could be found in~\cite{Caffarelli1995}, in which a polynomial pattern energy potential is used with a cut-off approach. On the other hand, for the Cahn-Hilliard equation~\eqref{CH equation-0} -- \eqref{CH-mu-0} with a singular Flory-Huggins energy potential~\eqref{CH energy}, an $L^\infty$ bound is automatically satisfied: $-1 < \phi < 1$, so that the PDE is well-defined. Meanwhile, in spite of such an automatic $L^\infty$ bound, a uniform distance between the solution away from the singular limit values will play a more important role, due to the singular nature in the nonlinear analysis. In fact, for the 2-D CH equation~\eqref{CH equation-0} -- \eqref{CH-mu-0}, the separation property has also been justified at a theoretical level~\cite{abels07, debussche95}, i.e., a uniform distance between the phase variable and the singular limit values ($-1$ and $1$) has been derived, dependent on $\varepsilon$, $\theta_0$ and the initial data. For the 3-D equation, a theoretical proof of the separation property has not been available, while we make such an assumption in this article, to facilitate the numerical iteration analysis. 

In addition, the system defined in~\eqref{CH energy} has a symmetric double-well structure. Notice that $\theta_0 >0$ is an $O (1)$ constant, and many interesting profiles could be obtained by the scientific computing with such a constant scale; see the detailed numerical simulation results in~\cite{chen22a}, with $\theta_0 =3$ and $\theta_0 =3.5$. A careful calculation reveals that, for $\theta_0 >1$, this free energy supports a spatially uniform equilibria solution: $\phi \equiv \pm \phi_* $, with $\phi_* \in (0,1)$ satisfying a steady-state equation: $\ln (1+\phi) - \ln (1-\phi) - \theta_0 \phi  = 0$. Of course, if the initial data does not have a mass average of $\pm \phi_*$, the PDE solution will not convergence to such a trivial steady-state solution, $\phi \equiv \pm \phi_*$. For the Allen-Cahn (AC) equation, the associated $L^2$ gradient flow, the separation property is satisfied with such a minima value of the double well, i.e., $- \phi_* \le \phi \le \phi_*$ at any time, provided that the initial data also satisfied this separation bound. Meanwhile, for the CH equation, this bound will not be satisfied, due to the fact that the maximum principle is not available any more for an $H^{-1}$ gradient flow. In more details, the phase separation constant $\epsilon_0$ for the 2-D Cahn-Hilliard equation, namely $-1 + \epsilon_0 \le \phi \le 1 - \epsilon_0$, depends on both $\theta_0$ and $\varepsilon$, as well as the initial separation constant, since the surface diffusion part has also played an important role in the separation estimate. Also see the related analysis in~\cite{LiD2021a}.  

On the other hand, the value of $\theta_0$ is fixed, with $2 \le \theta_0 \le 4$. For larger values of $\theta_0$ the equilibrium value increases towards the singular value of $\phi= \pm 1$. This may impact the numerical performance; however we do not investigate the effect in this work.  

In terms of the numerical design for the Cahn-Hilliard equation~\eqref{CH equation-0} -- \eqref{CH-mu-0} with logarithmic energy potential, the positivity preserving property generally posses the primary challenge~\cite{Fan17, jeong16, jeong15, LiX16, LiH2017, peng17a, peng17b, qiao14, yang19a}. Regarding a theoretical justification of the positivity-preserving property, a pioneering analysis was reported in \cite{elliott92a} in which the implicit Euler algorithm was applied combined with the finite element approximation in space. The positivity-preserving property is proved, while the unique solvability is theoretically justified under a condition for the time step size, which comes from the implicit discretization of the expansive term. To overcome this shortcoming, a convex splitting numerical scheme is proposed and analyzed in~\cite{chen19b} in which implicit treatment of the singular logarithmic and surface diffusion terms along with an explicit update of the linear expansive term was combined with the standard finite difference spatial approximation. The theoretical properties that have been established for the proposed numerical scheme include unconditional unique solvability, a positivity-preserving property, unconditional energy stability, and an optimal rate of convergence in the $\ell^\infty (0, T; H_h^{-1} ) \cap \ell^2 (0, T; H_h^1)$ norm. In particular, the singular and convex nature of the logarithmic term prevents the numerical solution from reaching the singular limit values of $\pm 1$, and this fact plays an essential role in the positivity-preserving analysis. Such an energy minimization analysis technique has been widely used in various gradient flows, including the phase field equation with Flory-Huggins potential~\cite{chen22b, chen22a, Dong2021a, Dong2022a, dong19b, dong20a, Yuan2021a, Yuan2022a}, the liquid film droplet model~\cite{ZhangJ2021}, the Poisson-Nernst-Planck system~\cite{LiuC2021a, LiuC2022a}, and the reaction-diffusion system~\cite{LiuC2021b, LiuC2022c, LiuC2022b}, etc. 

Although the theoretical analysis has been well-established for the first order convex splitting numerical scheme to the Cahn-Hilliard equation~\eqref{CH equation-0} -- \eqref{CH-mu-0} with Flory-Huggins energy potential, under the condition that it is exactly executed, the numerical implementation turns out to be highly challenging, due to the nonlinear and singular nature of the logarithmic terms involved in the numerical method. The focus of this paper will therefore be centered on the development and analysis of an iterative method for the numerical implementation of a first-order-in-time convex splitting numerical scheme to the Cahn-Hilliard equation~\eqref{CH equation-0} -- \eqref{CH-mu-0} with Flory-Huggins energy potential. For second-order (in time) numerical schemes, the positivity-preserving property and the modified energy stability have also been theoretically established, either in the BDF2 approach~\cite{chen19b} or in the Crank-Nicolson version~\cite{chen22a}, using similar theoretical techniques. However, the numerical implementation and the iteration analysis will be more involved and we reserve this for future work. A na\"{i}ve iterative approach may lead to a numerical solution not satisfying the positivity-preserving property in the iteration process. As an example, the full approximation storage (FAS) multi-grid method was applied in~\cite{chen19b} to implement the proposed scheme while the iteration convergence analysis for the FAS-like multi-grid method was established in~\cite{ChenL2020} for a convex optimization of a polynomial approximation energy potential. Although some convincing numerical results were reported in~\cite{chen19b} with a singular energy potential involved, a theoretical justification of such an iteration convergence analysis is not available. 

In this article, we propose and analyze an alternative iterative method, called the preconditioned steepest descent (PSD) solver, for the numerical implementation of the convex splitting numerical scheme to the Cahn-Hilliard equation~\eqref{CH equation-0} -- \eqref{CH-mu-0} with Flory-Huggins energy potential. The PSD solver for the p-Laplacian equation was considered in a pioneering work~\cite{Huang2007}, while an application of the PSD algorithm to a more general, regularized elliptic equation is analyzed in~\cite{feng2017preconditioned}, in which a much sharper iteration convergence rate has been established due to the higher order diffusion term involved. More applications of the PSD solver have been reported to various gradient flow models~\cite{ChenXC22a, cheng2019a, cheng2019d, cheng2022b, cheng2022a, fengW18a, fengW18b, ZhangJ2021}, etc. The robustness of this approach is demonstrated again in the numerical implementation of the algorithm to the Cahn-Hilliard equation with logarithmic energy potential, as described in~\cite{chen19b}. The key point is to use a linearized version of the nonlinear operator as a pre-conditioner to obtain a search direction. In other words, at each iteration stage, the surface diffusion operator and the $(-\Delta_h )^{-1}$ operator (for the temporal derivative) are kept the same as the original form, and a constant-coefficient linear operator is used to approximate the nonlinear part in the chemical potential expansion. In turn, the resulting equation for the search direction is efficiently solved with the help of FFT, since all the linear operators have eigenfunctions that are exactly the same as the Fourier basis functions. Afterward, with the search direction available, a one-parameter optimization of the corresponding numerical energy functional over the search direction is taken at the iteration stage. In fact, it is a strictly convex optimization in terms of the parameter, with singular and monotone logarithmic terms involved. Again, a careful positivity-preserving analysis ensures a unique solution of this one-parameter optimization, and the positivity of the logarithmic arguments are theoretically justified. Since it is a convex optimization, the Newton's iteration can be efficiently implemented and the positivity property will be preserved in the iteration process if the initial guess is sufficiently accurate. 

To verify the advantage of such a numerical solver, we present an iteration convergence analysis of the PSD iteration algorithm. Based on the fact that the equations can be reformulated as minimization problems involving strictly convex functionals in Hilbert spaces, the convexity analysis enables us to theoretically derive the convergence analysis for the nonlinear iterative solver. However, such an analysis is much more challenging than the gradient equations with a polynomial approximation of the energy potential since a positivity-preserving property must be justified at each iteration stage. Moreover, a uniform distance between the numerical solution and the singular limit values (of $\pm 1$), i.e.~the strict separation property, must be established to pass through the nonlinear estimates associated with the logarithmic terms. More specifically, the convergence estimate at the previous time step gives a discrete $H_h^1$ bound of the initial iteration error. Meanwhile, the non-increasing numerical energy (at each iteration stage) indicates a uniform discrete $\ell^2$ bound of the numerical solution in the iteration process. Furthermore, a careful application of a discrete Sobolev embedding establishes a connection between the discrete $\ell^2$ norm and the corresponding energy norm associated with the preconditioning stage. All these techniques lead to a theoretical justification of the geometric convergence rate for the PSD iteration solver. As a result, an $H_h^1$ convergence estimate for the iteration error leads to the strict separation property of the numerical solution at the next iteration stage, with the help of an inverse inequality. To our knowledge, this is the first result of a nonlinear iteration convergence for an iteration solver applied to a singular energy potential gradient flow.

The rest of this paper is organized as follows. In Section~\ref{sec:numerical scheme}, we review the finite difference spatial discretization and recall the convex splitting numerical scheme for the Cahn-Hilliard equation~\eqref{CH equation-0} -- \eqref{CH-mu-0} with Flory-Huggins energy potential. Some preliminary estimates are derived as well. In Section~\ref{sec: PSD}, the PSD iteration solver is proposed. In Section~\ref{sec:geometric-convergence}, a theoretical analysis of the geometric convergence rate, as well as the positivity-preserving analysis in the iteration process, is provided. Finally, some numerical results are presented in Section~\ref{sec:numerical results} and we provide concluding remarks in Section~\ref{sec:conclusion}.

\section{Review of the numerical scheme} 
\label{sec:numerical scheme}

	\subsection{The finite difference spatial discretization}
	\label{subsec:finite difference}

The spatial discretization notations are excerpted from~\cite{guo16, wise10, wise09}, and the references therein. We summarize the necessary notations below. For $\Omega = [0,1]^2$, and for any $N\in\mathbb{N}$, the mesh size is given by $h := \frac{1}{N}$, and it is assumed that the mesh spacing in the $x$ and $y$ directions are the same. Additionally, the following two uniform, infinite grids are introduced, with grid spacing $h>0$:
	\[
E := \{ p_{i+\hf} \ |\ i\in {\mathbb{Z}}\}, \quad C := \{ p_i \ |\ i\in {\mathbb{Z}}\},
	\]
where $p_i = p(i) := (i-\hf)\cdot h$. With these grids in place, we define three 2-D discrete $N^2$-periodic function spaces: 
	\begin{eqnarray*}
	\begin{aligned}
{\mathcal C}_{\rm per} &:= \left\{\nu: C \times C\rightarrow {\mathbb{R}}\ \middle| \ \nu_{i,j} = \nu_{i+\alpha N,j+\beta N}, \ \forall \, i,j,\alpha,\beta \in \mathbb{Z} \right\},
	\\
{\mathcal E}^{\rm x}_{\rm per} &:=\left\{\nu: E \times C\rightarrow {\mathbb{R}}\ \middle| \ \nu_{i+\frac12,j}= \nu_{i+\frac12+\alpha N,j+\beta N}, \ \forall \, i,j,\alpha,\beta \in \mathbb{Z}\right\} , 
    \\
{\mathcal E}^{\rm y}_{\rm per} &:=\left\{\nu: E \times C\rightarrow {\mathbb{R}}\ \middle| \ \nu_{i,j+\frac12}= \nu_{i+\alpha N,j+\frac12+\beta N}, \ \forall \, i,j,\alpha,\beta \in \mathbb{Z}\right\} , 
	\end{aligned}
	\end{eqnarray*}
in which the identification $\nu_{i,j} = \nu(p_i,p_j)$ was used. The functions of ${\mathcal C}_{\rm per}$ are called {\emph{cell centered functions}}. The functions of ${\mathcal E}^{\rm x}_{\rm per}$ and ${\mathcal E}^{\rm y}_{\rm per}$ are called {\emph{east-west}} and  {\emph{north-south  edge-centered functions}}, respectively. In addition, we define the space of mean zero functions as 
\begin{eqnarray}\label{def:mean-zero-space}
    \mathring{\mathcal C}_{\rm per}:=\left\{\nu\in {\mathcal C}_{\rm per} \ \middle| 0 = \overline{\nu} :=  \frac{h^2}{| \Omega|} \sum_{i,j=1}^m \nu_{i,j} \right\}.
\end{eqnarray} 
Finally, the space  $\vec{\mathcal{E}}_{\rm per} := {\mathcal E}^{\rm x}_{\rm per}\times {\mathcal E}^{\rm y}_{\rm per}$ is introduced. 

The spatial average and difference operators are given by 
	\begin{eqnarray*}
&& A_x \nu_{i+\hf,j} := \frac{1}{2}\left(\nu_{i+1,j} + \nu_{i,j} \right), \quad D_x \nu_{i+\hf,j} := \frac{1}{h}\left(\nu_{i+1,j} - \nu_{i,j} \right),
	\\
&& A_y \nu_{i,j+\hf} := \frac{1}{2}\left(\nu_{i,j+1} + \nu_{i,j} \right), \quad D_y \nu_{i,j+\hf} := \frac{1}{h}\left(\nu_{i,j+1} - \nu_{i,j} \right) , 
	\end{eqnarray*}
with $A_x,\, D_x: {\mathcal C}_{\rm per}\rightarrow{\mathcal E}_{\rm per}^{\rm x}$, $A_y,\, D_y: {\mathcal C}_{\rm per}\rightarrow{\mathcal E}_{\rm per}^{\rm y}$. Similarly, the following notations are introduced: 
	\begin{eqnarray*}
&& a_x \nu_{i, j} := \frac{1}{2}\left(\nu_{i+\hf, j} + \nu_{i-\hf, j} \right),	 \quad d_x \nu_{i, j} := \frac{1}{h}\left(\nu_{i+\hf, j} - \nu_{i-\hf, j} \right),
	\\
&& a_y \nu_{i,j} := \frac{1}{2}\left(\nu_{i,j+\hf} + \nu_{i,j-\hf} \right),	 \quad d_y \nu_{i,j} := \frac{1}{h}\left(\nu_{i,j+\hf} - \nu_{i,j-\hf} \right) , 
	\end{eqnarray*}
with $a_x,\, d_x : {\mathcal E}_{\rm per}^{\rm x}\rightarrow{\mathcal C}_{\rm per}$ and $a_y,\, d_y : {\mathcal E}_{\rm per}^{\rm y}\rightarrow{\mathcal C}_{\rm per}$. The discrete gradient $\nabh:{\mathcal C}_{\rm per}\rightarrow \vec{\mathcal{E}}_{\rm per}$ is defined as 
	\[
\nabh\nu_{i,j} := \left( D_x\nu_{i+\hf, j},  D_y\nu_{i, j+\hf} \right) ,
	\] 
and the discrete divergence $\nabh\cdot :\vec{\mathcal{E}}_{\rm per} \rightarrow {\mathcal C}_{\rm per}$ becomes 
	\[
\nabh\cdot\vec{f}_{i,j} := d_x f^x_{i,j}	+ d_y f^y_{i,j} ,
	\]
where $\vec{f} = (f^x,f^y)\in \vec{\mathcal{E}}_{\rm per}$. The standard 2-D discrete Laplacian, $\Delta_h : {\mathcal C}_{\rm per}\rightarrow{\mathcal C}_{\rm per}$, is given by 
	\begin{align*}
\Delta_h \nu_{i,j} := & \nabla_{h}\cdot\left(\nabla_{h}\phi\right)_{i,j} =  d_x(D_x \nu)_{i,j} + d_y(D_y \nu)_{i,j} 
	\\
= & \ \frac{1}{h^2}\left( \nu_{i+1,j}+\nu_{i-1,j}+\nu_{i,j+1}+\nu_{i,j-1} - 4\nu_{i,j,k}\right).
	\end{align*}
More generally, if $\mathcal{D}$ is a periodic \emph{scalar} function that is defined at all of the edge center points and $\vec{f}\in\vec{\mathcal{E}}_{\rm per}$, then $\mathcal{D}\vec{f}\in\vec{\mathcal{E}}_{\rm per}$, assuming point-wise multiplication, and we may define
	\[
\nabla_h\cdot \big(\mathcal{D} \vec{f} \big)_{i,j,k} = d_x\left(\mathcal{D}f^x\right)_{i,j}  + d_y\left(\mathcal{D}f^y\right)_{i,j}  .
	\]
Specifically, if $\nu\in \mathcal{C}_{\rm per}$, then $\nabla_h \cdot\left(\mathcal{D} \nabla_h  \ \ \right):\mathcal{C}_{\rm per} \rightarrow \mathcal{C}_{\rm per}$ is defined at a point-wise level: 
	\[
\nabla_h\cdot \big(\mathcal{D} \nabla_h \nu \big)_{i,j} = d_x\left(\mathcal{D}D_x\nu\right)_{i,j}  + d_y\left(\mathcal{D} D_y\nu\right)_{i,j}  .
	\]

Finally, we define the following grid inner products:   
	\begin{equation*}
	\begin{aligned}
\langle \nu , \xi \rangle &:= h^2 \sum_{i,j=1}^N  \nu_{i,j}\, \xi_{i,j},\quad \nu,\, \xi\in {\mathcal C}_{\rm per},\quad
&  
\\
 \eipx{\nu}{\xi} & := \langle a_x(\nu\xi) , 1 \rangle ,\quad \nu,\, \xi\in{\mathcal E}^{\rm x}_{\rm per} , 
 \\
 \eipy{\nu}{\xi} &:= \langle a_y(\nu\xi) , 1 \rangle ,\quad \nu,\, \xi\in{\mathcal E}^{\rm y}_{\rm per} ,  
\\
\Big\langle \vec{f}_1 , \vec{f}_2 \Big\rangle &: = \eipx{f_1^x}{f_2^x}	+ \eipy{f_1^y}{f_2^y} , \quad  \vec{f}_i = (f_i^x,f_i^y) \in \vec{\mathcal{E}}_{\rm per}, \ i = 1,2.
	\end{aligned}
	\end{equation*}	
The norms for cell-centered functions are accordingly introduced. If $\nu\in {\mathcal C}_{\rm per}$, then $\nrm{\nu}_2^2 := \langle \nu , \nu \rangle$; $\nrm{\nu}_p^p := \langle |\nu|^p , 1 \rangle$, for $1\le p< \infty$, and $\nrm{\nu}_\infty := \max_{1\le i,j\le N}\left|\nu_{i,j}\right|$. The gradient norms are similarly defined: for $\nu\in{\mathcal C}_{\rm per}$,
	\[
\nrm{ \nabla_h \nu}_2^2 : = \langle \nabh \nu , \nabh \nu \rangle = \eipx{D_x\nu}{D_x\nu} + \eipy{D_y\nu}{D_y\nu} ,
	\]
and for $1\le p<\infty$,
	\begin{equation*}
\nrm{\nabla_h \nu}_p := \left( \eipx{|D_x\nu|^p}{1} + \eipy{|D_y\nu|^p}{1}  \right)^{\frac1p} .
	\end{equation*}
In addition, a discrete $H_h^1$ norm is defined as:
\begin{eqnarray}\label{def:H1h-norm}
    \nrm{\nu}_{H_h^1}^2 : =  \nrm{\nu}_2^2+ \nrm{ \nabla_h \nu}_2^2.
\end{eqnarray}

	\begin{prop}  
	\label{lemma1}\label{lemma 1-0} 
Let $\mathcal{D}$ be an arbitrary periodic, scalar function defined on all of the edge center points. For any $\psi, \nu \in {\mathcal C}_{\rm per}$ and any $\vec{f}\in\vec{\mathcal{E}}_{\rm per}$, the following summation by parts formulas are valid: 
        \begin{subequations}\label{eq:summation-by-parts}
	\begin{align}
\langle \psi , \nabla_h\cdot\vec{f} \rangle & = - \langle \nabla_h \psi , \vec{f} \rangle ,
	\\
\langle \psi , \nabla_h\cdot \left(\mathcal{D}\nabla_h\nu\right) \rangle & = - \langle \nabla_h \psi , \mathcal{D}\nabla_h\nu \rangle .
        \end{align}
	\end{subequations}
	\end{prop} 

To facilitate the convergence analysis, we introduce a discrete analogue of the space $H_{per}^{-1}\left(\Omega\right)$, as outlined in~\cite{wang11a}. Suppose that $\mathcal{D}$ is a positive, periodic scalar function defined at all of the edge center points. For any $\phi\in{\mathcal C}_{\rm per}$, there exists a unique $\psi\in\mathring{\mathcal C}_{\rm per}$ that solves
	\begin{eqnarray*}
\mathcal{L}_{\mathcal{D}}(\psi):= - \nabla_h \cdot\left(\mathcal{D}\nabla_h \psi\right) = \phi - \overline{\phi} ,
	\end{eqnarray*}
where we recall that $\overline{\phi} := |\Omega|^{-1}\langle \phi, 1 \rangle$. We equip this space with a bilinear form: for any $\phi_1,\, \phi_2\in \mathring{\mathcal C}_{\rm per}$, define
	\begin{equation*}
\cipgen{ \phi_1 }{ \phi_2 }{\mathcal{L}_{\mathcal{D}}^{-1}} := \langle \mathcal{D}\nabla_h \psi_1 , \nabla_h \psi_2 \rangle ,
	\end{equation*}
where $\psi_i\in\mathring{\mathcal C}_{\rm per}$ is the unique solution to 
	\begin{equation*}
\mathcal{L}_{\mathcal{D}}(\psi_i):= - \nabla_h \cdot\left(\mathcal{D}\nabla_h \psi_i\right)  = \phi_i, \quad i = 1, 2.
	\end{equation*}

Since $\mathcal{L}_{\mathcal{D}}$ is symmetric positive definite, $\cipgen{ \ \cdot \ }{\ \cdot \ }{\mathcal{L}_{\mathcal{D}}^{-1}}$ is an inner product on $\mathring{\mathcal C}_{\rm per}$. (See~\cite{wang11a}.) When $\mathcal{D}\equiv 1$, we drop the subscript and write $\mathcal{L}_{1} = \mathcal{L}$ and, in this case, we write $\cipgen{ \ \cdot \ }{\ \cdot \ }{\mathcal{L}_{\mathcal{D}}^{-1}} =: \cipgen{ \ \cdot \ }{\ \cdot \ }{-1,h}$. In the general setting, the norm associated with this inner product is denoted as $\nrm{\phi}_{\mathcal{L}_{\mathcal{D}}^{-1}} := \sqrt{\cipgen{\phi }{ \phi }{\mathcal{L}_{\mathcal{D}}^{-1}}}$, for all $\phi \in \mathring{\mathcal C}_{\rm per}$. In particular, if $\mathcal{D}\equiv 1$, the notation becomes $\nrm{\, \cdot \, }_{\mathcal{L}_{\mathcal{D}}^{-1}} =: \nrm{\, \cdot \, }_{-1,h}$. 

\begin{prop}\label{prop:grad-minus1-split}
For any $\phi \in \mathring{\mathcal C}_{\rm per}$, we have 
\begin{eqnarray}
\nrm{\phi}_2 \le \nrm{\phi}_{-1,h}^{\frac12} \nrm{\nabla_h \phi}_2^{\frac12} .  \label{lem 3-0}
\end{eqnarray}
%with $B_0$ only dependent on $\Omega$, independent on $\varepsilon$. 
\end{prop}

\begin{proof}
The identity is based on the summation-by-parts formula,
	\begin{equation}
\cipgen{\phi_1 }{ \phi_2 }{\mathcal{L}_{\mathcal{D}}^{-1}} = \langle \phi_1 , \mathcal{L}_{\mathcal{D}}^{-1} (\phi_2) \rangle = \langle \mathcal{L}_{\mathcal{D}}^{-1} (\phi_1) , \phi_2 \rangle ,
	\end{equation}
and the definitions above.
	\end{proof}

\subsection{A positivity-preserving, energy stable numerical scheme} 

Consider a uniform partition of time, $0 = t_0 < t_1 < \cdots < t_F = T$, such that $t_k = k \Delta t$. The first order convex splitting scheme to the Cahn-Hilliard equation~\eqref{CH equation-0} -- \eqref{CH-mu-0}, with Flory-Huggins energy potential and a constant mobility ${\cal M} (\phi) \equiv 1$, that we consider herein was proposed in~\cite{chen19b} and is stated as follows: given $\phi^k\in {\mathcal C}_{\rm per}$, find  $\phi^{k+1} \in {\mathcal C}_{\rm per}$ such that
\begin{eqnarray}\label{eqn:scheme} 
  \begin{aligned} 
    \frac{\phi^{k+1} - \phi^k}{\dt} =&  \Delta_h  \mu^{k+1} ,     
\\
  \mu^{k+1} = & \ln (1+\phi^{k+1}) - \ln (1-\phi^{k+1}) - \theta_0 \phi^k - \varepsilon^2 \Delta_h \phi^{k+1} .      
\end{aligned} 
\end{eqnarray}

To define the initial conditions for the numerical scheme above, we let $\Phi$ be the exact solution of the Cahn-Hilliard equation given by~\eqref{CH equation-0} -- \eqref{CH-mu-0} and take the initial data to have sufficient regularity so that the exact solution has regularity of class $\mathcal{R}$, where 
        \begin{equation}
\mathcal{R} := H^2 \left(0,T; C_{\rm per}(\Omega)\right) \cap H^1 \left(0,T; C^2_{\rm per}(\Omega)\right) \cap L^\infty \left(0,T; C^6_{\rm per}(\Omega)\right),
	\label{assumption:regularity.1}
	\end{equation}	
i.e.~assume $\Phi \in \mathcal{R}$. Additionally, we suppose that $N = 2K+1$ and let ${\cal P}_N: C_{\rm per}(\Omega) \rightarrow {\cal B}^K(\Omega)$ denote the (spatial) Fourier projection operator, where ${\cal B}^K$ is the space of $\Omega$-periodic (complex) trigonometric polynomials of degree up to and including $K$. Furthermore, we define $\mathcal{P}_h:C_{\rm per}(\Omega) \rightarrow \mathcal{C}_{\rm per}$ as the canonical grid projection operator. Set $\Phi_N (\, \cdot \, ,t) := {\cal P}_N \Phi (\, \cdot \, ,t)$,  the (spatial) Fourier projection of the exact solution into ${\cal B}^K$. Then, using the mass-conservative projection for the initial data, $\phi^0 = {\mathcal P}_h \Phi_N (\, \cdot \, , t=0)$, that is
	\begin{equation}
\phi^0_{i,j} := \Phi_N (p_i, p_j, t=0) ,
	\label{initial data-0}
	\end{equation}	 
the positivity-preserving property, unique solvability, unconditional energy stability, and mass conservation for this scheme have been established in \cite{chen19b} and are summarized in the following theorem. 

	\begin{thm}  \cite{chen19b} 
	\label{CH-positivity} 
Given $\phi^0 = {\mathcal P}_h \Phi_N (\, \cdot \, , t=0)$ and $\phi^k \in\mathcal{C}_{\rm per}$, with $\| \phi^k \|_\infty \le M$, for some finite $M >0$,  and $\left|\overline{\phi^k}\right| < 1$, there exists a unique solution $\phi^{k+1}\in\mathcal{C}_{\rm per}$ to \eqref{eqn:scheme}, with $\phi^{k+1}-\overline{\phi^k}\in\mathring{\mathcal{C}}_{\rm per}$, $\| \phi^{k+1} \|_\infty < 1$, and
        \begin{equation} 
\overline{\phi^m} = \overline{\phi^{m-1}} ,  \quad \forall \ m \in \mathbb{N} .
	\label{mass conserv-2} 
	\end{equation} 
In addition, we have the following energy estimate,  
\begin{align} 
E_h(\phi^{k+1}) + \dt \| \nabla_h \mu^{k+1} \|_2^2 \le E_h(\phi^k) ,  
\end{align}
where
\begin{align}
E_h (\phi) := \langle 1 + \phi , \ln (1 + \phi)  \rangle + \langle 1 - \phi , \ln (1 - \phi)  \rangle
  - \frac{\theta_0}{2} \| \phi \|_2^2 + \frac{\varepsilon^2}{2} \| \nabla_h \phi \|_2^2 , 
    \label{CH-energy-0} 
\end{align} 
and the following uniform-in-time $H_h^1$ estimate,
	\begin{equation}
\| \nabla_h \phi^m \|_2 \le C_1 , \quad  \forall \, m \ge 0 , 
\label{CH-H1-1}    
	\end{equation}
where $C_1$ depends only on the initial data, $\Omega$, and $\varepsilon$. 
	\end{thm} 
	
Moreover, an optimal rate convergence estimate is available in the $\ell^\infty (0,T; H_h^{-1}) \cap \ell^2 (0,T; H_h^1)$ norm and is summarized in the theorem below. 

\begin{thm}  \cite{chen19b} 
	\label{thm:preliminary-convergence}
Suppose that the initial data satisfies $\Phi(\, \cdot \, ,t=0) \in C^6_{\rm per}(\Omega)$ and assume that the exact solution $\Phi$ for the Cahn-Hilliard equation~\eqref{CH equation-0} -- \eqref{CH-mu-0} is of regularity class $\mathcal{R}$. Then, for all positive integers $k$ with $t_k \le T$, there exists a constant $C_2>0$ that is independent of $k$, $\dt$, and $h$ such that
	\begin{equation}
\| \tilde{\phi}^k \|_{-1,h} +  \left( \varepsilon^2 \dt   \sum_{m=1}^{k} \| \nabla_h \tilde{\phi}^m \|_2^2 \right)^{1/2}  \le C_2 ( \dt + h^2 ),   
	\label{CH_LOG-convergence-0}
	\end{equation}
provided that $\dt$ and $h$ are sufficiently small and where the error grid function $\tilde{\phi}^m$ is defined as
        \begin{equation} 
\tilde{\phi}^m := \mathcal{P}_h \Phi_N^m - \phi^m ,  \quad \forall \ m \in \left\{ 0 ,1 , 2, 3, \cdots \right\} .
	\label{CH_LOG-error function-1}
	\end{equation}
	\end{thm}  

We conclude by remarking that the discrete norm $\nrm{ \, \cdot \, }_{-1,h}$ is well defined for the error grid function $\overline{\tilde{\phi}^m}$ since \eqref{initial data-0} implies that $\overline{\tilde{\phi}^m} =0$, for any $m \in \left\{ 0 ,1 , 2, 3, \cdots \right\}$.

\subsection{Strict separation property of the numerical solution and other preliminaries} 

In this subsection, we derive a strict separation property for the numerical solution of \eqref{eqn:scheme} provided that it has been exactly implemented. With this goal in mind, we present several useful properties for the exact solution of the Cahn-Hilliard equation given by~\eqref{CH equation-0} -- \eqref{CH-mu-0}. Specifically, if we suppose that the exact solution $\Phi$ for the Cahn-Hilliard equation~\eqref{CH equation-0} -- \eqref{CH-mu-0} is of regularity class $\mathcal{R}$, then we expect to have (and assume to have) the following separation property:
\begin{equation} 
 1+ \Phi, \  1 - \Phi  \ge \epsilon_0 , 
    \label{assumption:separation}
\end{equation}  
at a point-wise level, for some $\epsilon_0 > 0$. Note that $\epsilon_0$ is the uniform distance between the phase variable and the singular value limits that is dependent on $\varepsilon,  \theta_0$, and the initial data discussed in the Introduction above. Therefore, such a separation parameter $\epsilon_0$ is solely related to the PDE problem. Additionally, if $\Phi\in L^\infty(0,T;H^\ell_{\rm per}(\Omega))$, then the following projection approximation is standard:
	\begin{equation} 
\nrm{\Phi_N - \Phi}_{L^\infty(0,T;H^n)}  
   \le C_{\rm P} h^{\ell-n} \nrm{\Phi }_{L^\infty(0,T;H^\ell)},  \quad \forall \ 0 \le n \le \ell ,
	\label{projection-est-0} 
	\end{equation} 
where $C_{\rm P} > 0$ only depends on $\Omega$. Furthermore, $h$ can be chosen sufficiently small so that 
\begin{align*}
    1 + \Phi_N , \, 1 - \Phi_N \ge \left(\nicefrac{3}{4}\right) \epsilon_0.
\end{align*}
Finally, the following mass conservative property is available at the discrete level since $\Phi_N \in {\cal B}^K$: 
	\begin{equation} 
\overline{\Phi_N^m} = \frac{1}{|\Omega|}\int_\Omega \, \Phi_N ( \cdot, t_m) \, d {\bf x} = \frac{1}{|\Omega|}\int_\Omega \, \Phi_N ( \cdot, t_{m-1}) \, d {\bf x} = \overline{\Phi_N^{m-1}} ,  \quad \forall \ m \in\mathbb{N}, 
	\label{mass conserv-1} 
	\end{equation} 
where the notations $\Phi_N^m$, $\Phi^m$ denote $\Phi_N(\, \cdot \, , t_m)$ and $\Phi(\, \cdot \, , t_m)$, respectively.

\begin{lem}\label{lem:sep-num-sol}
    Let the initial data $\Phi(\, \cdot \, ,t=0) \in C^6_{\rm per}(\Omega)$ and suppose the exact solution for the Cahn-Hilliard equation~\eqref{CH equation-0} -- \eqref{CH-mu-0} is of regularity class $\mathcal{R}$. Additionally, suppose that the exact solution for the Cahn-Hilliard equation~\eqref{CH equation-0} -- \eqref{CH-mu-0} satisfies the separation property \eqref{assumption:separation}. Then, provided $\dt$ and $h$ are sufficiently small and we take a linear refinement of $\dt$ such that $C_L h \le \dt \le C_U h$, we have
\begin{equation} 
 1+ \phi^m, \, \, 1 - \phi^m  \ge \frac{\epsilon_0}{2} ,  
    \label{assumption:separation-1}
\end{equation}  	
at a point-wise level, for all positive integers $m$ such that $0 \le m \le k+1$.
\end{lem}

\begin{proof}
As a result of the leading order convergence estimate~\eqref{CH_LOG-convergence-0} and the linear refinement requirement, we obtain
\begin{equation} 
    \| \nabla_h \tilde{\phi}^m \|_2  \le \frac{C_2 (\dt + h^2 )}{\varepsilon \dt^\frac12}  
    \le C_3 (\dt^\frac12 + h^\frac32) ,   \quad 0 \le m \le k+1 .
    \label{CH_LOG-convergence-1}
	\end{equation}	
Meanwhile, the following inverse inequality is available for any function $f$ such that $\overline{f} =0$: 
\begin{equation} 
  \| f \|_\infty \le \frac{C_{\rm Inv} \| \nabla_h f \|_2}{h^{\delta_0}} ,  \quad \mbox{for } \delta_0 >0 .  
  \label{inverse ineq-1} 
\end{equation} 
Hence, we arrive at the following $\| \cdot \|_\infty$ estimate of the numerical error function, at each time step $t_m$: 
\begin{equation} 
    \| \tilde{\phi}^m \|_\infty \le \frac{C_{\rm Inv} \| \nabla_h \tilde{\phi}^m \|_2}{h^{\delta_0}} 
    \le \frac{C_{\rm Inv} C_3 (\dt^\frac12 + h^\frac32)}{h^{\delta_0}} 
    \le C_{\rm Inv} C_3 (\dt^\frac14 + h^\frac54) \le \frac{\epsilon_0}{4} ,   \quad 0 \le m \le k+1 , 
    \label{CH_LOG-convergence-2}
\end{equation}
for $0 < \delta_0 < \frac14$, provided that $\dt$ and $h$ are sufficiently small. The  combination of this inequality with the fact that $1 + \Phi_N , \, 1 - \Phi_N \ge \left(\nicefrac{3}{4} \right) \epsilon_0$, concludes the proof.
\end{proof}

Finally, the strict separation property of the numerical scheme allows us to obtain the following lemma which will be critical to the proof of the geometric convergence of the PSD solver presented in Section \ref{sec:geometric-convergence}.

\begin{lem}\label{lem:Rk-bound}
Let the initial data $\Phi(\, \cdot \, ,t=0) \in C^6_{\rm per}(\Omega)$ and suppose the exact solution $\Phi$ for the Cahn-Hilliard equation~\eqref{CH equation-0} -- \eqref{CH-mu-0} is of regularity class $\mathcal{R}$. Additionally, suppose that the exact solution for the Cahn-Hilliard equation~\eqref{CH equation-0} -- \eqref{CH-mu-0} satisfies the separation property \eqref{assumption:separation}. Let $\phi^k\in {\mathcal C}_{\rm per}$ and $\phi^{k+1}\in {\mathcal C}_{\rm per}$ be solutions to \eqref{eqn:scheme} at consecutive time steps. Define
\begin{align} 
  R_k :=& \, \frac{\varepsilon^2}{2} ( \| \nabla_h \phi^k \|_2^2  - \| \nabla_h \phi^{k+1} \|_2^2 )  
  + \langle 1 + \phi^k , \ln (1 + \phi^k)  \rangle 
   -  \langle 1 + \phi^{k+1} , \ln (1 + \phi^{k+1} )  \rangle   \nonumber 
\\
     &+ \langle 1 - \phi^k , \ln (1 - \phi^k)  \rangle 
   -  \langle 1 - \phi^{k+1} , \ln (1 - \phi^{k+1} )  \rangle     
  - \theta_0 \langle \phi^k , \phi^k - \phi^{k+1} \rangle  . \label{R_k-est-0} 
\end{align}   
Then, provided $\dt$ and $h$ are sufficiently small and a linear refinement of $\dt$ is taken such that $C_L h \le \dt \le C_U h$, we have 
    \begin{align}\label{R_k-est-5}
        R_k \le C_{\rm R}  (\dt^\frac12 + h^\frac32),
    \end{align}
where $C_{\rm R}$ depends on $\Omega, \varepsilon, \epsilon_0,$ and $\theta_0$ but does not depend on $k, h$ or $\dt$.
\end{lem}

\begin{proof}
To begin, we note that the regularity assumption for the exact solution $\Phi$ implies the following estimates for the projection solution $\Phi_N$: 
\begin{eqnarray} 
  \| \Phi_N^{k+1} - \Phi_N^k \|_{-1,h} \le C_4 \dt ,  \quad 
  \| \Phi_N^{k+1} - \Phi_N^k \|_2 \le C_5 \dt , \quad   
  \| \nabla_h ( \Phi_N^{k+1} - \Phi_N^k ) \|_2 \le C_6 \dt .  
  \label{CH_LOG-convergence-3}
\end{eqnarray}
Meanwhile, inequalities \eqref{CH_LOG-convergence-0} and \eqref{CH_LOG-convergence-1} yield
\begin{equation} 
    \| \tilde{\phi}^m \|_2 \le \| \tilde{\phi}^m \|_{-1, h}^\frac12 \cdot 
    \| \nabla_h \tilde{\phi}^m \|_2^\frac12  \le C_2^\frac12 C_3^\frac12 (\dt^\frac34 + h^\frac74) ,      
    \quad m = k, k+1 ,  
    \label{CH_LOG-convergence-4-2}
	\end{equation}	
where we have invoked Proposition \ref{prop:grad-minus1-split}. A combination of~\eqref{CH_LOG-convergence-3} with \eqref{CH_LOG-convergence-4-2} indicates that 
\begin{align}
    \| \phi^{k+1} - \phi^k \|_2 &\le  \left(C_5 + C_2^\frac12 C_3^\frac12\right) (\dt^\frac34 + h^\frac74)  , 
    \label{CH_LOG-convergence-5-1}
\end{align}
and
\begin{align}
  \| \nabla_h ( \phi^{k+1} - \phi^k ) \|_2 &\le \left(C_6 + C_3  \right) (\dt^\frac12 + h^\frac32) ,  
  \label{CH_LOG-convergence-5-2}
\end{align} 
provided $\dt < 1$ and where we have added and subtracted appropriate terms and invoked the triangle inequality.

Therefore, based on the Cauchy-Schwarz and triangle inequalities and the uniform-in-time $H_h^1$ estimate~\eqref{CH-H1-1}, the following estimate is available for the first term in the definition of $R_k$: 
\begin{align} 
    \| \nabla_h \phi^k \|_2^2  - \| \nabla_h \phi^{k+1} \|_2^2  
    &= \langle \nabla_h ( \phi^k + \phi^{k+1} )  , \nabla_h ( \phi^k - \phi^{k+1} )  \rangle  
    \nonumber 
\\
  &\le
      \| \nabla_h ( \phi^k + \phi^{k+1} ) \|_2 \cdot \| \nabla_h ( \phi^k - \phi^{k+1} )  \|_2     
      \nonumber 
\\
  &\le 2 C_1   \left(C_6 + C_3  \right) (\dt^\frac12 + h^\frac32) .
      %    \nonumber
      %    \\
      % &\le C \tilde{C}_1 \hat{C} (\dt^\frac12 + h^\frac32)  .      
     \label{R_k-est-1}
	\end{align}	
For the nonlinear inner product difference, an application of the intermediate value theorem gives  
\begin{eqnarray} 
   \langle 1 + \phi^k , \ln (1 + \phi^k)  \rangle 
   -  \langle 1 + \phi^{k+1} , \ln (1 + \phi^{k+1} )  \rangle   
   = \langle \ln ( 1 + \eta_1 ) +1 , \phi^k - \phi^{k+1} \rangle  , \label{R_k-est-2-1}   
\end{eqnarray}   
where $\eta_1$ is between $\phi^k$ and $\phi^{k+1}$. Meanwhile, by the strict separation property~\eqref{assumption:separation-1}, we see that 
\begin{eqnarray} 
   |  \ln ( 1 + \eta_1 ) +1 | \le \ln (2 \epsilon_0^{-1}) + 1 . \label{R_k-est-2-2}   
\end{eqnarray}   
In turn, a combination of~\eqref{R_k-est-2-1} and \eqref{R_k-est-2-2} along with the fact that $\| f \|_2 \le |\Omega|^{\frac12} \|f\|_\infty$, for any $f \in C_{\rm per}(\Omega)$, leads to
\begin{align} 
   \langle 1 + \phi^k , \ln (1 + \phi^k)  \rangle 
   -  \langle 1 + \phi^{k+1} , \ln (1 + \phi^{k+1} )  \rangle &\le | \Omega |^\frac12  \| \ln ( 1 + \eta_1 ) +1 \|_\infty \cdot  \| \phi^k - \phi^{k+1} \|_2 
   \nonumber
   \\
   &\le \left(C_5 + C_2^\frac12 C_3^\frac12\right) | \Omega |^\frac12 ( \ln (2 \epsilon_0^{-1}) + 1 )  
    (\dt^\frac34 + h^\frac74)  , \label{R_k-est-2-3}   
\end{align}  
with the convergence estimate~\eqref{CH_LOG-convergence-5-1} applied in the last step. Similarly, we are able to obtain 
\begin{align} 
   \langle 1 - \phi^k , \ln (1 - \phi^k)  \rangle 
   -  \langle 1 - \phi^{k+1} , \ln (1 - \phi^{k+1} )  \rangle    
   \le \left(C_5 + C_2^\frac12 C_3^\frac12\right) | \Omega |^\frac12 ( \ln (2 \epsilon_0^{-1}) + 1 )  
    (\dt^\frac34 + h^\frac74) . \label{R_k-est-3}   
\end{align}  
Finally, the last term on the right hand side of~\eqref{R_k-est-0} is bounded as follows: 
\begin{align} 
    -\theta_0 ( \phi^k , \phi^k - \phi^{k+1} )  
   &\le  \theta_0 \| \phi^k \|_2 \cdot \| \phi^k - \phi^{k+1} \|_2  \nonumber 
\\
  &\le \theta_0 | \Omega |^\frac12 \left(C_5 + C_2^\frac12 C_3^\frac12\right) (\dt^\frac34 + h^\frac74) , 
    \label{R_k-est-4}   
\end{align}   
in which the fact that $\| \phi^k \|_\infty < 1$ has been applied. Therefore, a substitution of~\eqref{R_k-est-1}, \eqref{R_k-est-2-3}, \eqref{R_k-est-3}, and \eqref{R_k-est-4} into~\eqref{R_k-est-0} results in the following bound, provided that $\dt$ and $h$ are sufficiently small:  
\begin{eqnarray} 
   R_k \le C_{\rm R}  (\dt^\frac12 + h^\frac32)  ,   
\end{eqnarray}
where $C_{\rm R}$ depends on $\Omega, \varepsilon, \epsilon_0,$ and $\theta_0$ but does not depend on $k, h$ or $\dt$.
\end{proof}

 \section{The preconditioned steepest descent iteration solver}
	\label{sec: PSD}

In this section, we present the preconditioned steepest descent iteration solver. For the numerical solution of (\ref{eqn:scheme}), we consider the discrete operator
\begin{eqnarray}
\mathcal{N}_h (\phi) :=  (-\Delta_h)^{-1} 
( \phi-\phi^k ) + \dt ( \ln (1+\phi) - \ln (1-\phi) ) - \varepsilon^2 \dt \Delta_h \phi, 
  \label{nonlinear equation-1} 
\end{eqnarray}
and we set
\begin{eqnarray}
    f= \theta_0 \dt \phi^k.
\end{eqnarray}
Hence, given $\phi^k\in {\mathcal C}_{\rm per}$, solving the numerical scheme \eqref{eqn:scheme} for $\phi^{k+1}\in {\mathcal C}_{\rm per}$ is equivalent to solving the following nonlinear system 
\begin{eqnarray}\label{eqn:nonlinearsys}
\mathcal{N}_h (\phi)=f,  \quad \mbox{for} \, \, \,  \phi \in {\mathcal C}_{\rm per} . 
\end{eqnarray} 

It should be understood that the analysis focuses on a single iteration of the numerical method \eqref{eqn:scheme} and we thus utilize the notation $\phi:= \phi^{k+1}$ throughout the remainder of the paper. 

\begin{lem}
    Let $\phi^k\in {\mathcal C}_{\rm per}$ be given. Define the discrete energy
\begin{eqnarray}\label{eqn:disconvexenergy}
J_h(\phi):= \frac12 \| \phi-\phi^k \|_{-1, h}^2 + \dt ( \langle 1 + \phi , \ln (1 + \phi ) \rangle 
 + \langle 1 - \phi , \ln (1 - \phi ) \rangle ) 
 +\frac{\varepsilon^2 \dt}{2} \| \nabla_h \phi \|_2^2 - \langle f , \phi \rangle,
\end{eqnarray} 
over the admissible set
\begin{align}\label{eqn:admissible-set}
    W_h:= \left\{\phi \in {\mathcal C}_{\rm per} \Big| \phi - \phi^k \in \mathring{\mathcal C}_{\rm per} \right\}.
\end{align} 
Then, solving (\ref{eqn:nonlinearsys}) is equivalent to minimizing $J_h(\phi)$.
\end{lem}

\begin{proof}
A direct calculation reveals that $J_h$ is twice Gataeux differentiable and convex. Specifically, we have
\begin{eqnarray}\label{eqn:1st}
\delta_\phi J_h (\phi)(v)= \langle ( - \Delta_h )^{-1} ( \phi-\phi^k ) 
 + \dt ( \ln (1+\phi) - \ln (1-\phi) ) , v \rangle  
 + \varepsilon^2 \dt \langle \nabla_h \phi , \nabla_h v \rangle 
 - \langle f , v \rangle ,
\end{eqnarray} 
for any $v \in \mathring{\mathcal C}_{\rm per}$, and
\begin{eqnarray}\label{eqn:2nd}
\delta_{\phi\phi} J_h (\phi)(v,w)= \langle  ( - \Delta_h )^{-1} 
v , w \rangle + \dt  \left\langle  \frac{1}{1+\phi} + \frac{1}{1-\phi} ,  v w  \right\rangle 
 + \varepsilon^2 \dt \langle \nabla_h v , \nabla_h w \rangle ,
\end{eqnarray} 
for any $v,w \in \mathring{\mathcal C}_{\rm per}$. This implies the convexity of $J_h$, due to the fact that 
\begin{eqnarray}
\delta_{\phi\phi} J_h (\phi)(v,v)= \langle  ( - \Delta_h )^{-1} 
v , v \rangle + \dt  \left\langle  \frac{1}{1+\phi} + \frac{1}{1-\phi} ,  v^2  \right\rangle 
 + \varepsilon^2 \dt \langle \nabla_h v , \nabla_h v \rangle  \ge 0.
\end{eqnarray}
Setting $\delta_\phi J_h (\phi)(v) = 0$ concludes the proof.
\end{proof}

In a steepest descent approach for finding the minimizer of $J_h(\phi)$, the general strategy is to find the normalized steepest descent direction $d_n \in \mathring{\mathcal C}_{\rm per}$ such that
    \begin{eqnarray}\label{eq:nsdd-def}
\delta_\phi J_h(v)(d_n) = - \nrm{\delta_\phi J_h(v)}_*, 
    \end{eqnarray}
under the restriction that
    \begin{eqnarray}
\nrm{d_n}_{\mathring{\mathcal C}_{\rm per}}=1,
    \end{eqnarray}
for all $v\in \mathring{\mathcal C}_{\rm per}$, where $ \nrm{\, \cdot \, }_{\mathring{\mathcal C}_{\rm per}}$ is a norm on $\mathring{\mathcal C}_{\rm per}$ and $ \nrm{\, \cdot \, }_*$ is the standard dual norm defined by 
\begin{eqnarray}
\nrm{\delta_\phi J_h(v)}_*:= \sup_{u \in \mathring{\mathcal C}_{\rm per}}\frac{|\delta_\phi J_h(v)(u)|}{\nrm{u}_{\mathring{\mathcal C}_{\rm per}}}.
\end{eqnarray}
However, the steepest descent method is not optimal. We therefore introduce preconditioning. Specifically, we introduce the operator $A_h$ such that
\begin{align}\label{def:Ak}
    A_h \psi:= (-\Delta_h)^{-1} \psi + \dt  \psi - \varepsilon^2 \dt \Delta_h \psi ,
\end{align}
for $\psi \in \mathring{\mathcal C}_{\rm per}$. This operator is clearly symmetric and positive definite. In fact, the standard steepest descent solver would lead to a very slow iteration convergence rate, especially for a high-dimensional optimization problem with a high condition number. A preconditioning approach, such as the one given by~\eqref{def:Ak}, provides a search direction much closer to the exact error than the standard gradient direction. In fact, among the operators involved in the nonlinear scheme~\eqref{nonlinear equation-1}, a linearized approach has to be taken. In particular, the temporal derivative and the surface diffusion parts correspond to constant-coefficient linear terms, so that the form of these two linear operators is exactly kept in the preconditioning process~\eqref{def:Ak}. Meanwhile, the nonlinear term $\dt ( \ln (1 + \phi) - \ln (1 - \phi))$ is monotone, and its linearized Lipschitz constant, which relies on its derivative, $\frac{\dt}{1 - \phi^2}$, is expected to be of $O (\dt)$, provided that the separation property is satisfied. In turn, we take a linear term, $\dt \psi$, to approximate the change associated with the nonlinear term. Such a combined choice will greatly improve the iteration convergence rate for a convex optimization with elliptic structure, as will be demonstrated by the theoretical analysis in the later sections. Also see the related analysis in~\cite{feng2017preconditioned}. 

With this operator at our disposal, the preconditioned steepest descent method is defined via the following algorithm.

\begin{algorithm}[H]
\caption{Preconditioned Steepest Descent Solver}\label{alg:psd}
\begin{algorithmic}
\State Define $\phi^{(0)}:= \phi^k$.
\State For $n \ge 0$, solve
\begin{align}\label{eqn:linearized}
    A_h d_n = f - \mathcal{N}_h (\phi^{(n)}),
\end{align}
for $d_n \in \mathring{\mathcal C}_{\rm per}$.
\State Determine step length $\bar{\alpha}$ via
\begin{align}\label{eqn:step-length}
\bar\alpha = {\rm argmin}_\alpha J_h(\phi^{(n)}+\alpha d_n) = {\rm argzero}_\alpha \Big\langle \delta_\phi J_h(\phi^{(n)}+\alpha d_n) , d_n \Big\rangle . 
\end{align} 
\State Set  
\begin{align}\label{eqn:psd}
\phi^{(n+1)}=\phi^{(n)}+\bar\alpha d_n . 
\end{align} 
\end{algorithmic}
\end{algorithm}

\begin{rem}
    Equation \eqref{eqn:linearized} can be exactly and efficiently implemented by an FFT-based finite difference solver.
\end{rem}

\begin{rem}
    It is observed that $J_h$ is a strictly convex function, and the convex energy terms, $\langle 1 + \phi , \ln (1 + \phi) \rangle$, $\langle 1 - \phi , \ln (1 - \phi) \rangle$, have singular and monotone derivatives as $\phi \searrow -1$ and $\phi \nearrow 1$. In turn, the one-parameter function in \eqref{eqn:step-length}, namely $\Big\langle \delta_\phi J_h(\phi^{(n)}+\alpha d_n) , d_n \Big\rangle$, is strictly convex in terms of $\alpha$, and this one-parameter function has singular and monotone derivatives as $(\phi^{(n)}+\alpha d_n ) \to \pm 1$. As a result, with an application of the positivity-preserving analysis technique reported in~\cite{chen19b}, there is a unique solution to this one-dimensional optimization problem, with $-1 < \phi^{(n+1)} < 1$, at a point-wise level. 
\end{rem}

\section{The geometric convergence of the preconditioned steepest descent solver} \label{sec:geometric-convergence}

We will now show that the preconditioned steepest descent solver has a geometric convergence rate. The proof of this fact requires the following lemmas and corollary.

\begin{lem}
The search direction $d_n$ defined in (\ref{eqn:linearized}) is the steepest descent direction, at $\phi^{(n)} \in W_h$ with respect to the norm $\norm{\, \cdot \, }{A_h}$, where
    \begin{eqnarray}
\nrm{u_n}_{A_h}^2= \nrm{u_n}_{-1,h}^2 + \dt \nrm{ u_n}_2^2+ \varepsilon^2 \dt \nrm{\nabla_h u_n}_2^2 .
    \end{eqnarray}
\end{lem}

\begin{proof}
The proof follows similarly to that found in \cite{ChenXC22a}. By definition, the normalized steepest decent direction at the point $v \in W_h$ is a vector $d \in \mathring{\mathcal C}_{\rm per}$ satisfying \eqref{eq:nsdd-def}. From \eqref{eqn:linearized}, we have 
\begin{eqnarray}
 \langle A_h d_n , v \rangle = -\delta_\phi J_h(\phi^{(n)})(v),
\end{eqnarray}
for all $v \in \mathring{\mathcal C}_{\rm per}$. Note that $d_n$ is the Riesz representation of the functional $\delta_\phi J_h(\phi^{(n)})$ in the space $\mathring{\mathcal C}_{\rm per}$ with respect to the norm $\norm{\, \cdot \, }{A_h}$. Hence  
 \begin{equation}
\nrm{d_n}_{A_h}=\nrm{\delta_\phi J_h(\phi^{(n)})}_*,
\end{equation} 
and 
 \begin{align}
\delta_\phi J_h(\phi^{(n)})(d_n)&=-\nrm{d_n}_{A_h}^2 
 =-\nrm{\delta_\phi J_h(\phi^{(n)})}_* \cdot \nrm{d_n}_{A_h},\label{eq:star-Ak-split}
\end{align} 
for all $d_n\in \mathring{\mathcal C}_{\rm per}$.
\end{proof}

\begin{cor}\label{cor:descent}
Let $\{\phi^{(n)}\}$ be the sequence generated by (\ref{eqn:psd}). Then we have 
    \begin{eqnarray}
J_h(\phi^{(n+1)})\leq J_h(\phi^{(n)}).
    \end{eqnarray}
\end{cor}

\begin{lem}\label{lem:lemma4}
For any $u, v \in \mathring{\mathcal C}_{\rm per}$, the following inequality is valid  
\begin{eqnarray}
\langle \delta _\phi J_h(u)-\delta _\phi J_h(v), u-v \rangle \geq C_{\rm LB} \nrm{u-v}_{A_h}^2 , 
\label{lem 4-0}
\end{eqnarray}
with $C_{\rm LB} = \min \left\{ \frac12, \varepsilon \dt^{-\frac12} \right\}$.
\end{lem}

\begin{proof}
A careful calculation reveals that 
\begin{align} 
  \langle \delta _\phi J_h(u)-\delta _\phi J_h(v), u-v \rangle &=
  \dt \langle \ln (1+u) - \ln (1+v) - \ln (1-u) + \ln (1-v)  , u-v \rangle  
  \nonumber
\\
  &\quad+ \| u - v \|_{-1, h}^2 + \varepsilon^2 \dt \| \nabla_h (u-v) \|_2^2   .   
  \label{lem 4-1} 
\end{align} 
Meanwhile, the following estimates are available by the convexity of the logarithmic terms: 
\begin{eqnarray}  
  \langle \ln (1+u) - \ln (1+v)  , u-v \rangle  \ge 0 , 
\end{eqnarray}
  and
\begin{eqnarray}
  \langle - \ln (1-u) + \ln (1-v)  , u-v \rangle  \ge 0 . 
   \label{lem 4-2} 
\end{eqnarray}    
As a consequence, we get 
\begin{align}  
  \langle \delta _\phi J_h(u)-\delta _\phi J_h(v), u-v \rangle &\ge  \| u - v \|_{-1, h}^2   
  + \varepsilon^2 \dt \nrm{\nabla_h (u-v)}_2^2  
  \nonumber 
\\
  &\ge \frac12 \| u - v \|_{-1, h}^2 + \frac12 \varepsilon^2 \dt \nrm{\nabla_h (u-v)}_2^2 + \varepsilon \sqrt{\dt} \| u-v \|_2^2 ,  \label{lem 4-3} 
\end{align} 
in which we have utilized Proposition \ref{prop:grad-minus1-split} to obtain 
\begin{align} 
  \frac12 \| u - v \|_{-1, h}^2 + \frac12 \varepsilon^2 \dt \nrm{\nabla_h (u-v)}_2^2  
  &\ge  \varepsilon \sqrt{\dt}  \| u - v \|_{-1,h} \cdot \nrm{\nabla_h (u-v)}_2  
    \nonumber
\\
  &\ge \varepsilon \sqrt{\dt} \| u-v \|_2^2 . 
  \label{lem 4-4} 
\end{align}     
In comparison with the form of $\nrm{u-v}_{A_h}^2$: 
\begin{eqnarray} 
  \nrm{u-v}_{A_h}^2 = \| u - v \|_{-1, h}^2  + \dt \nrm{u-v}_2^2 + \varepsilon^2 \dt \nrm{\nabla_h (u-v)}_2^2  ,  \label{lem 4-5} 
\end{eqnarray} 
we conclude that estimate \eqref{lem 4-0} is valid by choosing $C_{\rm LB} = \min \left\{ \frac12, \varepsilon \dt^{-\frac12} \right\}$.   
\end{proof}

\begin{lem}\label{lem:lemma5}
Let $\{\phi^{(n)}\}$ be the sequence generated by \eqref{eqn:psd}. Furthermore, suppose that 
\begin{equation} 
   1 + \phi^{(n)} \ge \frac{\epsilon_0}{4} ,  \quad 
   1 - \phi^{(n)} \ge \frac{\epsilon_0}{4} , 
   \label{assmp:strict-sep}  
\end{equation} 
at a point-wise level and define the iteration error $e_n:=J_h(\phi^{(n)})-J_h(\phi)$. Then we have 
\begin{align}
e_n \le \langle \delta _\phi J_h(\phi^{(n)})-\delta _\phi J_h(\phi), \phi^{(n)}-\phi \rangle \le C_{\rm UB}  \nrm{\delta_\phi J_h(\phi^{(n)})}_*^2, \label{lem 5-0-1} 
    \end{align}
and
    \begin{align}
 \left| \delta_{\phi\phi} J_h(\theta^n)(d_n,d_n) \right|  \le C_{\rm D2} \nrm{d_n}_{A_h}^2 , \label{lem 5-0-2} 
 \end{align}
for any $\theta^n$ in the line segment from $\phi^{(n)}$ to $\phi^{(n+1)}$, where the constants $C_{\rm UB}$, $C_{\rm D2}$ have the following forms: 
\begin{align} 
  C_{\rm UB} = C_{\rm LB}^{-1} = \max \left\{ 2, \varepsilon^{-1} \dt^{\frac12}  \right\} ,  \quad \text{and} \quad
  C_{\rm D2} = 1 + 4 \epsilon_0^{-1} \varepsilon^{-1} \dt^{\frac12} . 
  \label{lem 5-0-3}  
\end{align}  
\end{lem}

\begin{proof} 
By the properties of the convexity, we have
\begin{eqnarray}
e_n = J_h(\phi^{(n)})-J_h(\phi) \leq \langle \delta _\phi J_h(\phi^{(n)})-\delta _\phi J_h(\phi), \phi^{(n)}-\phi \rangle.
\end{eqnarray}
Using the fact that $\delta_\phi J_h(\phi) \equiv 0$, combined with an application of inequality \eqref{lem 4-0} (from Lemma~\ref{lem:lemma4}) and Young's inequality, we arrive at 
\begin{align}
&\langle \delta _\phi J_h(\phi^{(n)})-\delta _\phi J_h(\phi), \phi^{(n)}-\phi \rangle 
\nonumber \\
&\qquad= \langle \delta _\phi J_h(\phi^{(n)}) , \phi^{(n)}-\phi \rangle
\nonumber\\
%&\leq&\norm{ \delta_\phi J_h(\phi^{n,k})}_*\norm{\phi-\phi^{n,k}}_{1,4}
&\qquad\le \nrm{ \delta_\phi J_h(\phi^{(n)})}_*\nrm{\phi-\phi^{(n)}}_{A_h}
\nonumber
\\
&\qquad\le \frac12 C_{\rm LB}^{-1} \nrm{ \delta_\phi J_h(\phi^{(n)})}_*^2+\frac{1}{2} C_{\rm LB} \nrm{\phi-\phi^{(n)}}_{A_h}^2
\nonumber\\
&\qquad\le \frac12 C_{\rm LB}^{-1} \nrm{ \delta_\phi J_h(\phi^{(n)})}_*^2+ \frac12
(\delta _\phi J_h(\phi^{(n)})-\delta _\phi J_h(\phi), \phi^{(n)}-\phi).
\end{align}
Therefore, we can take constant $C_{\rm UB} = C_{\rm LB}^{-1} = \max \left\{ 2, \varepsilon^{-1} \dt^{\frac12}  \right\}$, such that
\begin{eqnarray}
e^n \leq \langle \delta _\phi J_h(\phi^{(n)})-\delta _\phi J_h(\phi), \phi^{(n)}-\phi \rangle \leq C_{\rm UB}  \nrm{ \delta_\phi J_h(\phi^{(n)})}_*^2.
\end{eqnarray} 

Next we derive an estimate for $\left| \delta_{\phi\phi} J_h(\theta^n)(d_n,d_n) \right|$. We begin by applying the discrete H\"older inequality to \eqref{eqn:1st} and \eqref{eqn:2nd} to obtain the following bounds:
\begin{align} 
|\delta_\phi J_h (\phi)(v)| &\leq   \nrm{\phi}_{-1, h} 
 \cdot \nrm{v}_{-1, h} + S_0 \dt | \Omega |^\frac12 \nrm{v}_{2}+  \varepsilon^2 \dt \nrm{\nabla_h \phi}_2 \cdot \nrm{\nabla_h v}_2 + \nrm{f}_{2} \cdot \nrm{v}_{2}, \label{eqn:1stbound} 
\end{align}
and
\begin{align}
|\delta_{\phi\phi}J_h (\phi)(v,w)| &\leq \nrm{v}_{-1, h} \cdot \nrm{w}_{-1, h}  + S_1 \dt \nrm{v}_{2} \cdot \nrm{w}_{2} + \varepsilon^2 \dt \nrm{\nabla_h v}_2 \cdot \nrm{\nabla_h w}_2, \label{eqn:2ndbound}
\end{align} 
where
\begin{align}
  S_0 = \nrm{ \ln \left(\frac{1+\phi}{1-\phi} \right) }_\infty , \quad \text{and} \quad S_1= \nrm{ \frac{2}{1- \phi^2} }_\infty  .  \label{eqn:bound-4}  
\end{align}
With the strict separation assumption \eqref{assmp:strict-sep} at hand, the following bounds for $S_0$ and $S_1$ become available:  
\begin{eqnarray} 
  S_0 = \nrm{ \ln \frac{1+ \phi^{k+1,(n)} }{1- \phi^{k+1,(n)}} }_\infty 
  \le \ln (4 \epsilon_0^{-1} ) ,   \quad \text{and} \quad
  S_1 = \nrm{ \frac{2}{1- ( \phi^{k+1,(n)} )^2} }_\infty  
  \le 8 \epsilon_0^{-1} .  \label{a priori-2}  
\end{eqnarray} 
Thus, with an application of \eqref{eqn:2ndbound}, we get  
\begin{eqnarray}
  \left| \delta_{\phi\phi} J_h(\theta^n)(d_n,d_n) \right| 
   \le \nrm{d_n}_{-1, h}^2 + S_1 \dt \nrm{d_n}_2^2 + \varepsilon^2 \dt \nrm{\nabla_h d_n}_2^2 . \label{lem 5-1} 
\end{eqnarray}

On the other hand, an application of the discrete Sobolev inequality (\ref{lem 3-0}) (in Proposition~\ref{prop:grad-minus1-split}) indicates that 
\begin{eqnarray} 
    \nrm{d_n}_{-1, h}^2 + \varepsilon^2 \dt \nrm{\nabla_h d_n}_2^2 \ge 2 \varepsilon \dt^\frac12   \nrm{d_n}_{-1,h} \cdot \nrm{\nabla_h d_n}_2  
    \ge 2 \varepsilon \dt^\frac12   \nrm{d_n}_2^2 . 
    \label{lem 5-3}
\end{eqnarray}  
Consequently, a substitution of (\ref{lem 5-3}) into (\ref{lem 5-1}) yields 
\begin{align}
  \left| \delta_{\phi\phi} J_h(\theta^n)(d_n,d_n) \right| 
   &\le  \nrm{d_n}_{-1, h}^2  + \varepsilon^2 \dt \nrm{\nabla_h d_n}_2^2    
    + S_1 \dt \nrm{d_n}_2^2   
     \nonumber\\
  &\le \left( 1 + 4 \epsilon_0^{-1} \varepsilon^{-1} \dt^{\frac12}  \right)   
   ( \nrm{d_n}_{-1,h}^2 + \varepsilon^2 \dt \nrm{\nabla_h d_n}_2^2 )  .    
   \label{lem 5-4-1} 
\end{align}
In comparison with the form of $\nrm{d_n}_{A_h}^2$: 
\begin{eqnarray} 
  \nrm{d_n}_{A_h}^2 = \| d_n \|_{-1, h}^2  + \dt \nrm{d_n}_2^2 + \varepsilon^2 \dt \nrm{\nabla_h d_n}_2^2  ,  \label{lem 5-4-2} 
\end{eqnarray} 
we conclude that estimate (\ref{lem 5-0-2}) is valid by choosing $C_{\rm D2} = 1 + 4 \epsilon_0^{-1} \varepsilon^{-1} \dt^{\frac12}$. %Note that both $B_0$ and $C_1$ are $\varepsilon$, $\dt$ independent. This finishes the proof of Lemma~\ref{lem:lemma5}. 
\end{proof}

\begin{rem} 
  We see that $C_{\rm UB} = O (1)$ if $\dt = O (\varepsilon^2)$, while $C_{\rm UB} = O (\varepsilon^{-1} \dt^{\frac12})$ with a small $\varepsilon$ value. A similar scaling law is available for $C_{\rm D2}$. Specifically, $C_{\rm D2} = O (\epsilon_0^{-1})$ if $\dt = O (\varepsilon^2)$, while $C_{\rm D2} = O (\epsilon_0^{-1} \varepsilon^{-1} \dt^{\frac12})$ with a small $\varepsilon$ value.
\end{rem}

\begin{thm} \label{theorem:convergence}
Under the separation assumption~\eqref{assumption:separation} for the exact PDE solution, let $\{\phi^{(n)}\}$ be the sequence generated by \eqref{eqn:psd}. Furthermore, it is assumed  that 
\begin{equation} 
   1 + \phi^{(n)} \ge \frac{\epsilon_0}{4} ,  \quad 
   1 - \phi^{(n)} \ge \frac{\epsilon_0}{4} , 
\end{equation} 
at a point-wise level. Then we have 
\begin{eqnarray}
e_n \leq \left(1 - \frac{1}{2 C_{\rm UB} C_{\rm D2}} \right)^n e_0 , \quad \text{with} \quad \frac{1}{2 C_{\rm UB} C_{\rm D2}} < 1.  
 \label{error convergence-0}
\end{eqnarray}
\end{thm}
\begin{proof}
From the definition of the steepest descent direction, we apply \eqref{lem 5-0-2} (from Lemma~\ref{lem:lemma5}) and get the following inequality, for an arbitrary $\alpha$: 
\begin{align}
J_h(\phi^{(n)}+\alpha d_n)-J_h(\phi^{(n)})&= \alpha\delta_\phi J_h(\phi^{(n)})(d_n) +\frac{\alpha^2}{2} \delta_{\phi\phi} J_h(\theta)(d_n,d_n)
\nonumber\\
&\le \alpha\delta_\phi J_h(\phi^{(n)})(d_n) +\frac{\alpha^2}{2} C_{\rm D2} \nrm{d_n}_{A_h}^2
\nonumber\\
&=\big(-\alpha+\frac{\alpha^2}{2}C_{\rm D2} \big)\nrm{ \delta_\phi J_h(\phi^{(n)})}^2_*.
\end{align}
Hence, the minimum is achieved at $\bar{\alpha}=\frac{1}{C_{\rm D2}}$ and we have 
\begin{align}
e_{n+1}-e_{n} 
&= J_h(\phi^{(n+1)})-J_h(\phi^{(n)})
\nonumber\\
&\le J_h(\phi^{(n)}+\bar{\alpha}d_n)-J_h(\phi^{(n)})
\nonumber\\
&= -\frac{1}{2C_{\rm D2}} \nrm{ \delta_\phi J_h(\phi^{(n)})}^2_*.
\end{align}
Therefore,
\begin{align}
e_{n}-e_{n+1} &\ge  \frac{1}{2C_{\rm D2}} \nrm{ \delta_\phi J_h(\phi^{(n)})}^2_*
\nonumber\\
&\ge  \frac{1}{2C_{\rm UB}C_{\rm D2}} e_n, 
\end{align}
in which the estimate \eqref{lem 5-0-1} of Lemma~\ref{lem:lemma5} was applied in the last step. Hence,
\begin{eqnarray}
e_{n+1}  \le  \left(1- \frac{1}{2C_{\rm UB}C_{\rm D2}} \right) e_n, \label{error-contraction} 
\end{eqnarray}
and the desired result follows. 
\end{proof}

\begin{rem} \label{rem: parameters} 
  A geometric convergence rate is assured by Theorem~\ref{theorem:convergence}. Regarding the convergence constant, we observe that $C_{\rm UB} C_{\rm D2} = O (\epsilon_0^{-1})$ for a time step choice of $\dt = O (\varepsilon^2)$, while $C_{\rm UB} C_{\rm D2} = O (\epsilon_0^{-1} \varepsilon^{-2} \dt)$ with a small $\varepsilon$ value. In turn, this estimate leads to a convergence rate of $\alpha_0^n$, with $\alpha_0 = 1 - O (\dt^{-1} \epsilon_0^{-1} \varepsilon^{\frac52})$ such that $0 < \alpha_0 < 1$ for $\dt = O (\varepsilon^2)$. 
  
  This analysis also verifies the following well-known fact observed in the extensive numerical experiments: the steepest descent nonlinear iteration provides a fast convergence for a small time step size. Specifically, with a smaller value of $\varepsilon$, the numerical implementation becomes more and more challenging. Fortunately, the choice of a small time step size also accelerates the convergence speed for a small value of $\varepsilon$.  
  
In fact, because of a general estimate $C_{\rm UB} C_{\rm D2} = O (\epsilon_0^{-1} \varepsilon^{-2} \dt)$, we see that an $O (1)$ geometric convergence rate is ensured if the time step size satisfies a constraint $\dt \varepsilon^{-2}  = O (\epsilon_0) \ll 1$. Meanwhile, for the first order scheme~\eqref{eqn:scheme}, the $\ell^\infty (0, T; H_h^{-1}) \cap \ell^2 (0, T; H_h^1)$ convergence analysis, as presented in~\cite{chen19b}, only requires a constraint $\dt \le \varepsilon^2$. This constraint is milder than the $O (1)$ geometric iteration convergence rate requirement, since the convexity analysis has greatly helped the error estimate. On the other hand, if an $\ell^\infty (0, T; \ell^2) \cap \ell^2 (0, T; H_h^2)$ error analysis is derived for the numerical scheme, a more severe time step constraint would be needed in the estimates, due to the complicated nonlinear expansion structure. Also see the related work~\cite{LiuC2021a}, in which the $\ell^\infty (0, T; \ell^2)$ error estimate has been presented for the Poisson-Nernst-Planck system.  
\end{rem}

\begin{rem} 
If the first order numerical scheme~\eqref{eqn:scheme} could be exactly implemented, the positivity-preserving, energy stability and optimal rate convergence analysis would be unconditional, i.e., the convergence estimate~\eqref{CH_LOG-convergence-0} would be always valid, and there is no constraint for the time step size $\dt$. Meanwhile, in terms of the PSD iteration solver to implement the numerical algorithm~\eqref{eqn:scheme}, the above estimates imply that, the iteration convergence rate will be greatly accelerated with an additional constraint $\dt \le O (\epsilon_0 \varepsilon^2)$. However, even if such an additional constraint is not satisfied, the iteration estimate~\eqref{error convergence-0} still indicates a geometric convergence rate for the PSD iteration, although the convergence speed will not be as good as the one with the additional constraint $\dt \le O (\epsilon_0 \varepsilon^2)$. Extensive numerical experiments have revealed that, five to ten iteration stages would be sufficient for the implementation in most practical computational examples, with reasonable physical parameters and time step sizes.   
\end{rem}

The contraction estimate \eqref{error convergence-0} is valid for the error of the discrete energy (\ref{eqn:disconvexenergy}). Meanwhile, such a contraction estimate is not directly available for the numerical error associated with the phase variable at the $k+1$ time step: $q_n := \phi^{(n)} - \phi$. However, we are still able to derive a geometric convergence estimate for such a numerical error. As in the previous section, we define $\phi:= \phi^{k+1}$. 

\begin{thm} \label{theorem:convergence-num-error}
Let the initial data $\Phi(\, \cdot \, ,t=0) \in C^6_{\rm per}(\Omega)$ and suppose the exact solution $\Phi$ for the Cahn-Hilliard equation~\eqref{CH equation-0} -- \eqref{CH-mu-0} is of regularity class $\mathcal{R}$. Additionally, suppose that the exact solution for the Cahn-Hilliard equation~\eqref{CH equation-0} -- \eqref{CH-mu-0} satisfies the separation property \eqref{assumption:separation}. Let $\phi^{(n)}$ be the sequence generated by \eqref{eqn:psd} and define the numerical error associated with the phase variable at the $(k+1)$-th time step to be $q_n := \phi^{(n)} - \phi$. Then, for any $n\ge 0$, provided $\dt$ and $h$ are sufficiently small and we take a linear refinement of $\dt$ such that $C_L h \le \dt \le C_U h$, it follows that
\begin{eqnarray}
    \| \nabla_h q_{n} \|_2^2 \le  \frac{2 C_{\rm R} }{\varepsilon^2}\left(1- \frac{1}{2C_{\rm UB}C_{\rm D2}} \right)^{n}  (\dt^\frac12 + h^\frac32) ,
\end{eqnarray}
and
\begin{eqnarray}
    \| q_{n} \|_{-1,h}^2 \le \left(1- \frac{1}{2C_{\rm UB}C_{\rm D2}} \right)^{n}  (\dt^\frac32 + h^\frac32),
\end{eqnarray}
which implies that
\begin{eqnarray}
    \| q_{n} \|_{H^1_h}^2 \le \left( \frac{\varepsilon^2 + 2 C_{\rm R} }{2\varepsilon^2} \right)\left(1- \frac{1}{2C_{\rm UB}C_{\rm D2}} \right)^{n}  (\dt^\frac12 + h^\frac32), \label{error-phi-0} 
\end{eqnarray}
with
\begin{eqnarray*}
    \frac{1}{2 C_{\rm UB} C_{\rm D2}} < 1.
\end{eqnarray*}
\end{thm}
\begin{proof}
The proof proceeds by induction. The base case with $n=0$ follows from inequalities \eqref{CH_LOG-convergence-5-1} and \eqref{CH_LOG-convergence-5-2} and the fact that $\Delta t, h > 0$ are chosen small enough. 

Now, let 
\begin{eqnarray}
    \| \nabla_h q_{n} \|_2^2 \le  \frac{2 C_{\rm R} }{\varepsilon^2}\left(1- \frac{1}{2C_{\rm UB}C_{\rm D2}} \right)^{n}  (\dt^\frac12 + h^\frac32) ,
\end{eqnarray}
and
\begin{eqnarray}
    \| q_{n} \|_{-1,h}^2 \le \left(1- \frac{1}{2C_{\rm UB}C_{\rm D2}} \right)^{n}  (\dt^\frac32 + h^\frac32),
\end{eqnarray}
so that
\begin{eqnarray}
    \| q_{n} \|_{H^1_h}^2 \le \left( \frac{\varepsilon^2 + 2 C_{\rm R} }{2\varepsilon^2} \right)\left(1- \frac{1}{2C_{\rm UB}C_{\rm D2}} \right)^{n}  (\dt^\frac12 + h^\frac32),
\end{eqnarray}
for some $n>0$. As a result, the following iteration error estimate becomes available: 
\begin{eqnarray} 
  \| \nabla_h q_{n} \|_2  \le   \sqrt{2} \varepsilon^{-1} C_{\rm R}^\frac12  (\dt^\frac14 + h^\frac34)  , \label{error-phi-3}   
\end{eqnarray}   
in which the fact that $\left(1- \frac{1}{2C_{\rm UB}C_{\rm D2}} \right) <1$ has been applied. Its substitution into the discrete inverse inequality~\eqref{inverse ineq-1} leads to 
\begin{equation} 
  \| q_{n} \|_\infty \le \frac{C_{\rm Inv} \| \nabla_h q_{n} \|_2}{h^{\delta_0}} 
  \le   \sqrt{2} \varepsilon^{-1} C_{\rm Inv}  C_{\rm R}^\frac12  (\dt^\frac18 + h^{\frac12})   
  \le  \frac{\epsilon_0}{4} , 
  \label{error-phi-4} 
\end{equation} 
for $\delta_0 < \frac18$, and provided that $\dt$ and $h$ are sufficiently small. A combination of~\eqref{error-phi-4} and the strict separation property~\eqref{assumption:separation-1} (for the exact numerical solution of~\eqref{eqn:scheme}, at $m=k+1$) results in  
	\begin{equation} 
 1+ \phi^{(n)}, \, \, 1 - \phi^{(n)}  \ge \frac{\epsilon_0}{4} ,  \quad \mbox{at a point-wise level} .  
    \label{a priori-3}
	\end{equation}  	
Additionally, the following functional inequality is available:
\begin{align} 
J_h(\phi^{(n+1)})-J_h(\phi) &= \delta_\phi J_h(\phi)(q_{n+1}) +\frac{1}{2} \delta_{\phi\phi} J_h(\beta)(q_{n+1}, q_{n+1}) 
\nonumber\\
&= \frac{1}{2} \delta_{\phi\phi} J_h(\beta)(q_{n+1} ,q_{n+1} )  
\nonumber 
\\
  &\ge \frac12 \| q_{n+1} \|_{-1,h}^2 + \frac{\varepsilon^2 \dt}{2} \| \nabla_h q_{n+1} \|_2^2 , 
   \label{error-phi-1} 
\end{align}  
with $\beta$ in the line segment from $\phi^{(n)}$ to $\phi$. Note that the second step comes from the facts that $\delta_\phi J_h(\phi) \equiv 0$ and
\begin{eqnarray} 
   \left\langle  \frac{1}{1+ \theta} + \frac{1}{1- \theta} ,  q_{n+1}^2  \right\rangle  \ge 0 . 
\end{eqnarray} 
As a direct consequence, a combination of Lemma \ref{lem:lemma5} and Theorem \ref{theorem:convergence} yields 
\begin{align} 
  \frac12 \| q_{n+1} \|_{-1,h}^2 + \frac{\varepsilon^2 \dt}{2} \| \nabla_h q_{n+1} \|_2^2 &\le e_{n+1} 
  \nonumber \\ 
  &\le \left(1- \frac{1}{2C_{\rm UB}C_{\rm D2}} \right)^{n+1} e_0  
  \nonumber 
\\
  &\le
     \left(1- \frac{1}{2C_{\rm UB}C_{\rm D2}} \right)^{n+1}  \Big( - \frac12 \| \phi-\phi^k \|_{-1, h}^2 
  + \dt R_k \Big)  
  \nonumber
  \\
  &\le \dt \left(1- \frac{1}{2C_{\rm UB}C_{\rm D2}} \right)^{n+1} R_k  , \label{error-phi-2}   
\end{align}   
where $R_k$ has been defined in Lemma \ref{lem:Rk-bound} and we have used the fact that 
\begin{align*}
    - \langle f, \phi^{(o)} \rangle + \langle f, \phi \rangle  &= \langle f, \phi^{k+1} \rangle - \langle f, \phi^{k} \rangle
    \\
    &= \langle f, \phi^{k+1} - \phi^k \rangle
    \\
    &= \theta_0 \dt \langle \phi^k, \phi^{k+1} - \phi^k \rangle . 
\end{align*} 
Hence, 
\begin{eqnarray}
\| \nabla_h q_{n+1} \|_2^2 \le  \frac{2 C_{\rm R} }{\varepsilon^2}\left(1- \frac{1}{2C_{\rm UB}C_{\rm D2}} \right)^{n+1}  (\dt^\frac12 + h^\frac32) ,
\end{eqnarray}
and 
\begin{eqnarray}
    \| q_{n+1} \|_{-1,h}^2 \le \left(1- \frac{1}{2C_{\rm UB}C_{\rm D2}} \right)^{n+1}  (\dt^\frac32 + h^\frac32) ,
\end{eqnarray}
with
\begin{eqnarray*}
    \frac{1}{2 C_{\rm UB} C_{\rm D2}} < 1.
\end{eqnarray*}
An application of Proposition \ref{prop:grad-minus1-split} and the definition of the discrete $H^1$ norm concludes the proof. 
\end{proof}

\begin{rem}
We observe that, although the preliminary estimate~\eqref{R_k-est-5} gives $R_k = O (\dt^\frac14 + h^\frac34)$ at a theoretical level, the practical computations indicate that $R_k = O (\dt)$, since $\phi$ is the exact numerical solution $\phi^{k+1}$ for the convex splitting scheme, so that the iteration convergence could be accelerated.  
\end{rem}

\begin{rem} 
For simplicity of presentation, we only provide the analysis for the 2-D Cahn-Hilliard equation~\eqref{CH equation-0} with Flory-Huggins energy potential. For the 3-D gradient flow, the strict separation property is an open problem even for the PDE solution, and such a theoretical issue poses a great challenge in the associated numerical analysis. Meanwhile, if the strict separation property is available for the 3-D PDE solution, the iteration convergence analysis and the strict separation estimate for the related numerical solver could be derived in a similar manner. 
\end{rem} 

\begin{rem} 
For the sake of brevity, we only consider a constant mobility, ${\cal M} (\phi) \equiv 1$. If a $\phi$-dependent mobility function is involved, the iteration convergence estimate and positivity-preserving analysis for the corresponding PSD solver could be derived in a careful way, following the techniques presented in~\cite{ChenXC22a} to deal with a gradient flow with polynomial approximation potential. The technical details are left to interested readers. 
\end{rem} 

%\begin{rem} 
%\textcolor{red}{This article only considers a first order numerical scheme, namely~\eqref{eqn:scheme}, for simplicity of presentation. If the first order numerical scheme~\eqref{eqn:scheme} is exactly implemented, the positivity-preserving analysis has already been theoretically justified in \cite{chen19b}. Meanwhile, the numerical implementation} \textcolor{blue}{along with the related analysis turns out to be a very challenging issue, and this is the focus of this article.}
%
%\textcolor{red}{
%For the second order (in time) numerical schemes, the positivity-preserving property and the modified energy stability have also been theoretically established, either in the BDF2 approach~\cite{chen19b} or in the Crank-Nicolson version~\cite{chen22a}, using similar theoretical techniques. } \textcolor{blue}{However, the numerical implementation and the iteration analysis will be more involved and we reserve this for future work.} 
%\end{rem} 

\begin{rem} 
The periodic boundary condition is considered in this article, for simplicity of presentation. Meanwhile, if a homogeneous Neumann boundary condition is imposed for the CH equation~\eqref{CH equation-0} -- \eqref{CH-mu-0}, given by   
\begin{equation} 
  \partial_{\n} \phi = 0 , \quad \partial_{\n} \mu = 0 ,  \qquad \mbox{on $\partial \Omega$} , 
  \label{BC-Neumann-1} 
\end{equation} 
the positivity-preserving, energy stability and optimal rate convergence analysis for the numerical scheme~\eqref{eqn:scheme} could be derived in the same manner. In fact, with a discrete approximation to the homogeneous Neumann boundary condition, the summation by parts formulas take the same form as the ones with a periodic boundary condition. The Sobolev interpolation inequality is also valid with a physical boundary condition, so that all the theoretical results become available. See the related works~\cite{chen16, diegel16, yan17} for the CH equation with a homogeneous Neumann boundary condition, in a polynomial approximation in the energy potential. In addition, the iteration convergence analysis for the PSD solver could also be established if the physical boundary condition is imposed, following similar ideas. 
\end{rem}

\section{Numerical results}  \label{sec:numerical results}

\subsection{Convergence test for the numerical scheme}

In this subsection we perform a numerical accuracy check for the numerical scheme~\eqref{eqn:scheme}, implemented by the proposed PSD iteration solver. The computational domain is chosen as $\Omega = [0,1]^2$, and the exact profile for the phase variable is set to be
	\begin{equation}
\Phi (x,y,t) = \frac{1}{\pi} \sin( 2 \pi x) \cos(2 \pi y) \cos( t) .
	\label{AC-1}
	\end{equation} 
With the choice of this exact profile, it is clear that the quantities $1 + \Phi$ and $1 - \Phi$ stay positive at a point-wise level, so that a uniform distance is available between the PDE solution and the singular limit values of $\pm 1$. Of course, to force $\Phi$ to satisfy the original PDE \eqref{CH equation-0} -- \eqref{CH-mu-0}, we must add an artificial, time-dependent forcing term. In turn, the numerical scheme~\eqref{eqn:scheme} is implemented to solve for (\ref{CH equation-0}), using the proposed PSD iteration. 

First, we verify the efficiency and accuracy of the proposed PSD iterative solver. The first time step is taken into consideration, and we take the spatial resolution as $N=256$ (with $h=\frac{1}{256}$). The discrete $\ell^\infty$ and $\ell^2$ iteration errors are displayed in Figure~\ref{fig: iteration_1}, in terms of the iteration number, if we take the time step size as $\dt=0.01$, and the interface width parameter as $\varepsilon=0.05$. The geometric convergence rate has been clearly observed in the iteration process, which justifies the theoretical analysis~\eqref{error-phi-0}. In fact, such an iteration has reached the machine precision within 20 iteration stages. In the practical computations, only 5 to 10 iteration stages are needed at each time step. 

Moreover, to investigate the iteration performance and its dependence on certain parameters, such as the time step size $\dt$ and interface width $\varepsilon$, we record the number of iterations to reach the machine precision (so that the discrete $\ell^2$ error is less than $10^{-15}$). In more details, the left plot of Figure~\ref{fig: iteration_2} displays the number of iterations in terms of $\varepsilon=0.01:0.01:0.1$, with a fixed $\dt=0.01$, while the right plot displays that in terms of $\dt=0.01:0.01:0.1$, with a fixed $\varepsilon=0.05$. In all these numerical tests, only 5 to 10 iteration stages are needed to reach a machine precision. Meanwhile, such a number of iteration will be reduced from 6 to 5 with an increase of $\varepsilon$, or with a decrease of the value of $\dt$. This numerical behavior also agrees with the analysis outlined in Remark~\ref{rem: parameters}.

	\begin{figure}
		\begin{center}
	 %\hbox{  	
\includegraphics[width=3.0in]{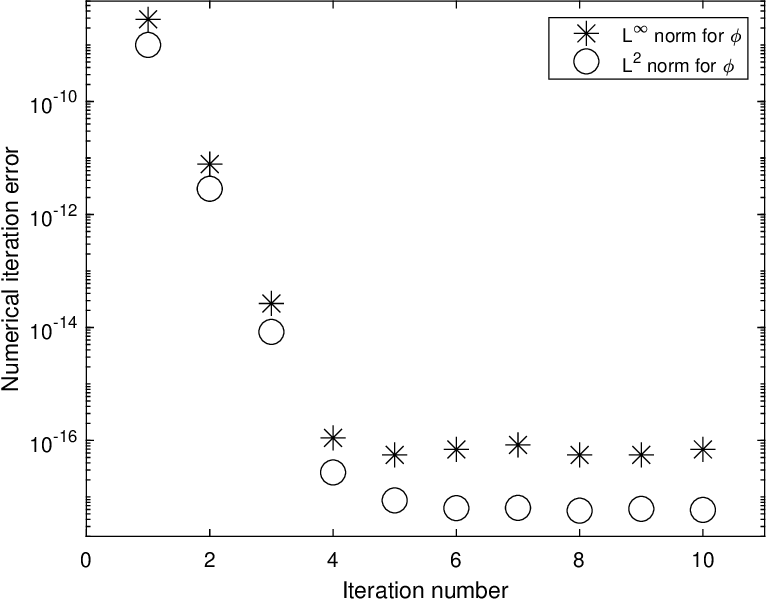}  %} 
	\end{center}	
\caption{The discrete $\ell^\infty$ and $\ell^2$ numerical errors vs. the iteration number, with a spatial resolution $N=256$. The numerical results are obtained by the proposed PSD iteration solver. The time step size and surface diffusion parameters are taken as: $\dt=0.01$, $\varepsilon=0.05$.}
	\label{fig: iteration_1}
	\end{figure}

	\begin{figure}
		\begin{center}
	 \hbox{  	
\includegraphics[width=3.0in]{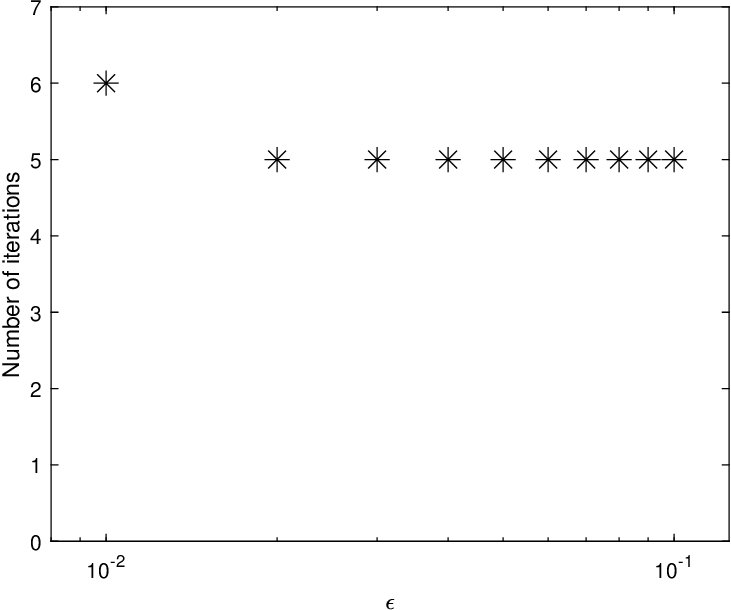}  \hskip 0.2cm
\includegraphics[width=3.0in]{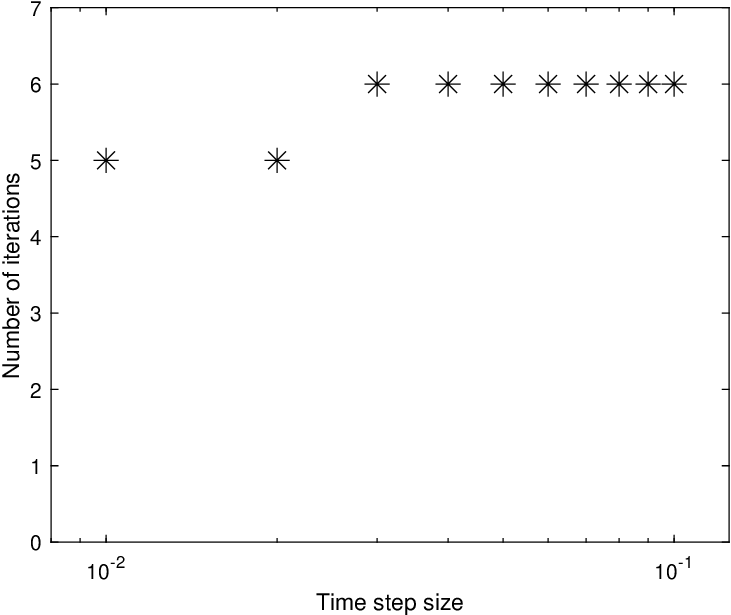} } 
	\end{center}	
\caption{Left: The number of iterations needed to obtain a machine precision for the PSD solver, in terms of $\varepsilon=0.01:0.01:0.1$, with a fixed $\dt=0.01$. Right:  The number of iterations needed to obtain a machine precision for the PSD solver, in terms of $\dt=0.01:0.01:0.1$, with a fixed $\varepsilon=0.05$.}
	\label{fig: iteration_2}
	\end{figure}

Of course, the accuracy test for the fully implemented numerical scheme is also very important. We fix the spatial resolution as $N=512$ (with $h=\frac{1}{512}$), so that the spatial numerical error is negligible. The final time is set as $T=1$, and the surface diffusion parameter is given by $\varepsilon=0.5$, while the expansive parameter is set as $\theta_0 = 2$. 
%and we set the artificial regularization parameter as $A=1$.
A sequence of time step sizes are taken as $\dt=\frac T{N_T}$ with $N_T = 100:100:1000$.  The expected temporal numerical accuracy assumption $e=C \dt$ indicates that $\ln |e|=\ln (CT) - \ln N_T$, so that we plot $\ln |e|$ versus $\ln N_T$ to demonstrate the temporal convergence order. The fitted line displayed in Figure~\ref{fig1} shows an approximate slope of $-0.9938$, which in turn verifies a nice first order temporal convergence in both the discrete $\ell^2$ and $\ell^\infty$ norms. 

	\begin{figure}
	\begin{center}
\includegraphics[width=3.0in]{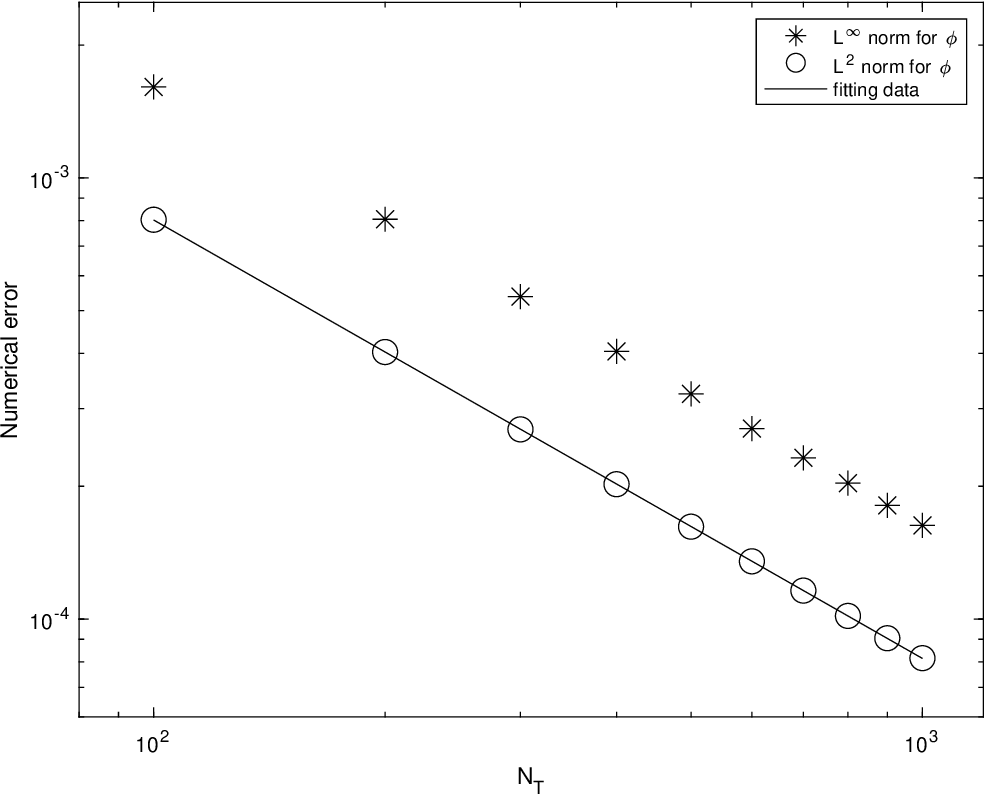}
	\end{center}
\caption{The discrete $\ell^2$ and $\ell^\infty$ numerical errors versus temporal resolution  $N_T$ for $N_T = 100:100:1000$, with a spatial resolution of $N=512$. The numerical results are obtained by the computation using the proposed PSD iteration solver to the numerical scheme~\eqref{eqn:scheme}. The surface diffusion parameter is taken to be $\varepsilon=0.5$, and the expansive parameter is set as $\theta_0 = 2$. The data lie roughly on curves $CN_T^{-1}$ for appropriate choices of $C$, confirming the full first order accuracy of the scheme.}
	\label{fig1}
	\end{figure}

\subsection{Numerical simulation of coarsening processes} 

In this subsection, a two-dimensional numerical simulation of the coarsening process is presented. The computational domain is set as $\Omega =[0,1]^2$, the expansive parameter is chosen to be $\theta_0=3$, and the interface width parameter is taken as $\varepsilon=0.005$. Meanwhile, a random initial data is chosen: 
	\begin{equation}
	\label{initial data-2}
\phi^0_{i,j}= 0.1 + 0.05 \cdot (2r_{i,j}-1), \quad \mbox{$r_{i,j}$ are uniformly distributed random numbers in $[0, 1]$} .  
	\end{equation}
	%cheng2020b, ChenN16, Promislow1991
Such a random initial data contains a wide spectrum of wave lengths in the Fourier expansion, so that many interesting structures will be observed in the long time simulation, in comparison with a smooth initial data. Meanwhile, although such a random initial data is only of $L^2 (\Omega)$ regularity, the constant-coefficient surface diffusion term would create a smooth solution within a short time interval, due to the parabolic nature of the PDE. Also see the related works~\cite{ChenN16, cheng2020b, Promislow1991}, in which a local-in-time Gevrey regularity (with real analytic regularity) solution has been established for certain gradient flow models, even if the initial data is only of $H^1$ or $H^2$ regularity. Extensive numerical experiments~\cite{chen22a} have also indicated a smooth solution profile after a very short initial time interval, with a random initial data. As a result, all the theoretical analysis in this article is expected to be valid in this numerical simulation.  

Again, the proposed PSD iteration solver is applied to implement the numerical scheme~\eqref{eqn:scheme} in this simulation. In the coarsening process, increasing values of $\dt$ are taken in the time evolution: $\dt = 5 \times 10^{-5}$ on the time interval $[0,1]$, $\dt = 10^{-4}$ on the time interval $[1, 3]$, $\dt = 2 \times 10^{-4}$ on the time interval $[3, 7]$, and $\dt = 5 \times 10^{-4}$ on the time interval $[7, 15]$. Whenever a new time step size is applied, we initiate the two-step numerical scheme by  taking $\phi^{-1} = \phi^0$, with the initial data $\phi^0$ given by the final time output of the last time period. The time snapshots of the evolution with $\varepsilon=0.005$ are displayed in Figure~\ref{fig3}, with significant coarsening observed in the system.  At the earlier time steps, many small structures are present.  At the final time, $T= 15$, a single structure emerges, and further coarsening is not possible. 

%\textcolor{red}{(SMW: Can we display the max and min values for $\phi$? They should be very close to $\pm 1$ given the large value for $\theta_0$.)}

\begin{figure}[h]
	\begin{center}
		\begin{subfigure}{0.48\textwidth}
			\includegraphics[height=0.48\textwidth,width=0.48\textwidth]{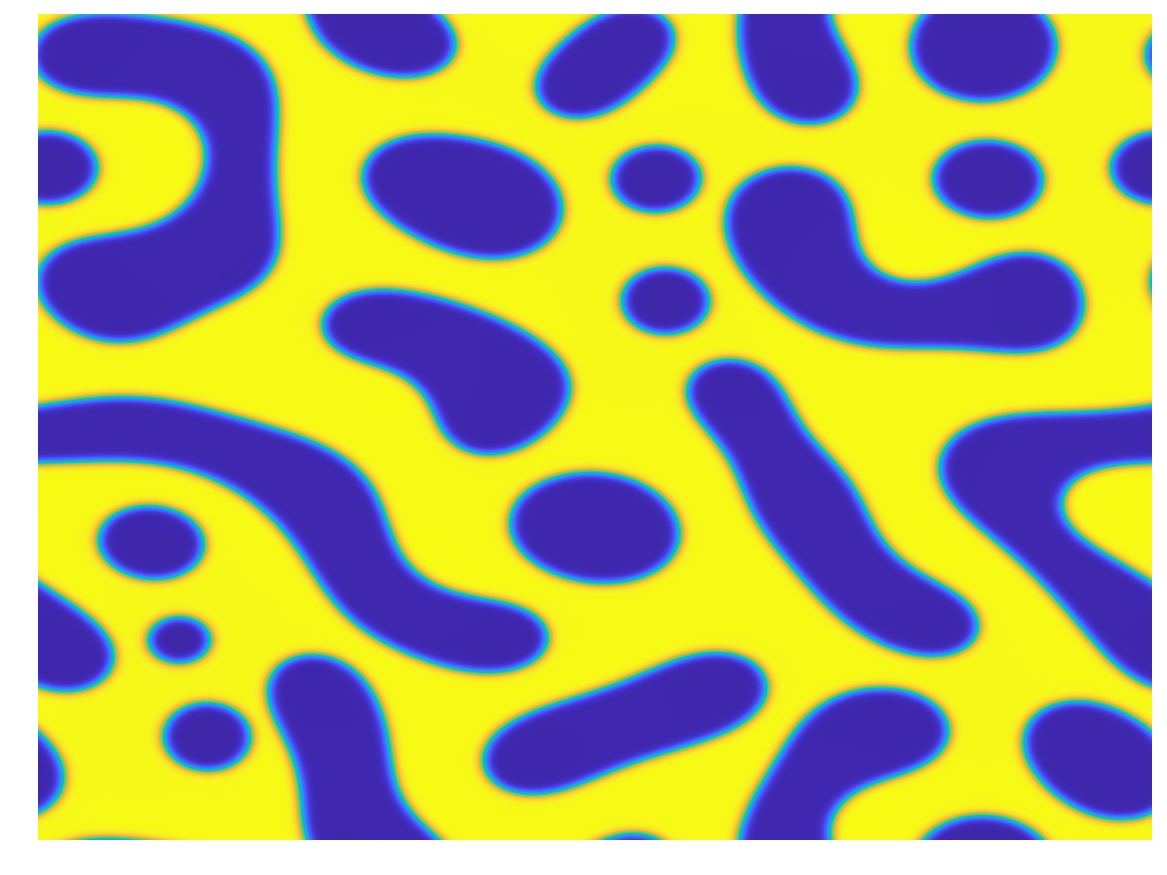}
			\includegraphics[height=0.48\textwidth,width=0.48\textwidth]{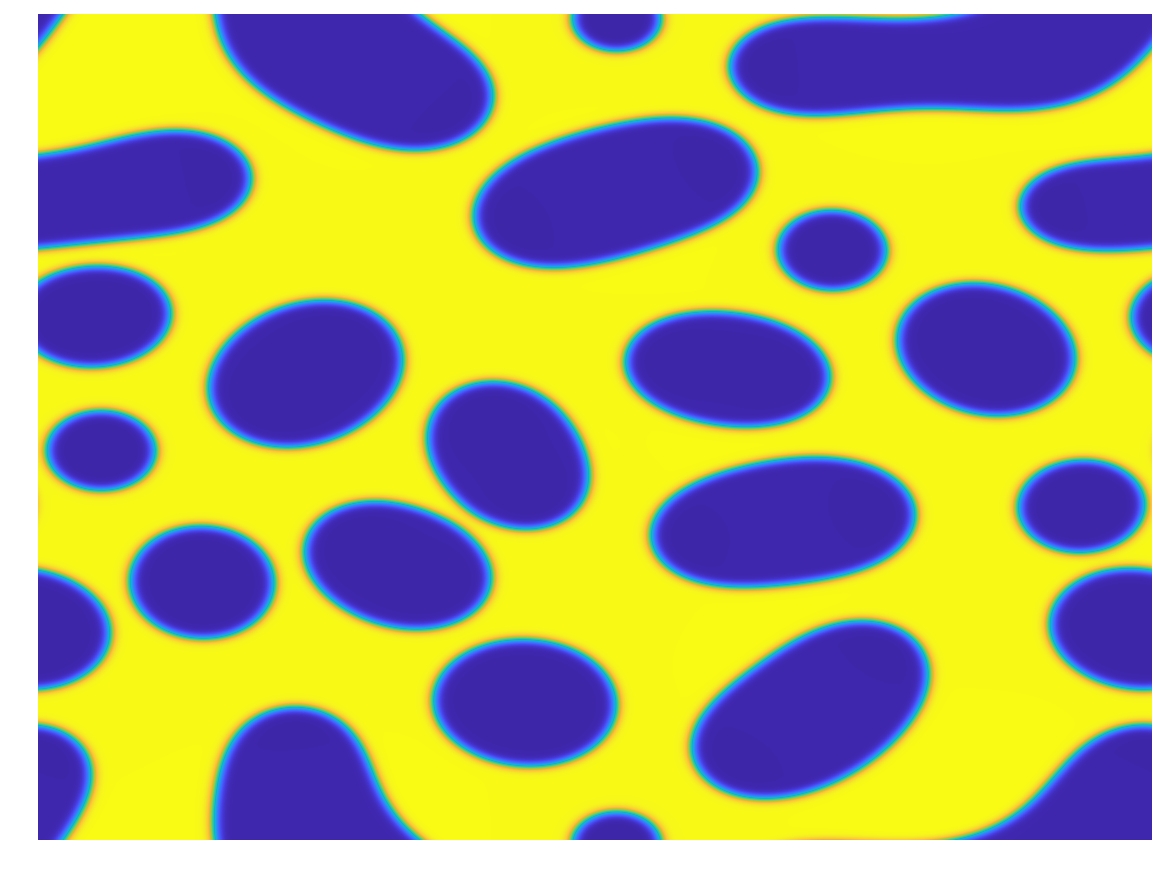}
			\caption*{$t=0.05, 0.1$}
		\end{subfigure}
		\begin{subfigure}{0.48\textwidth}
			\includegraphics[height=0.48\textwidth,width=0.48\textwidth]{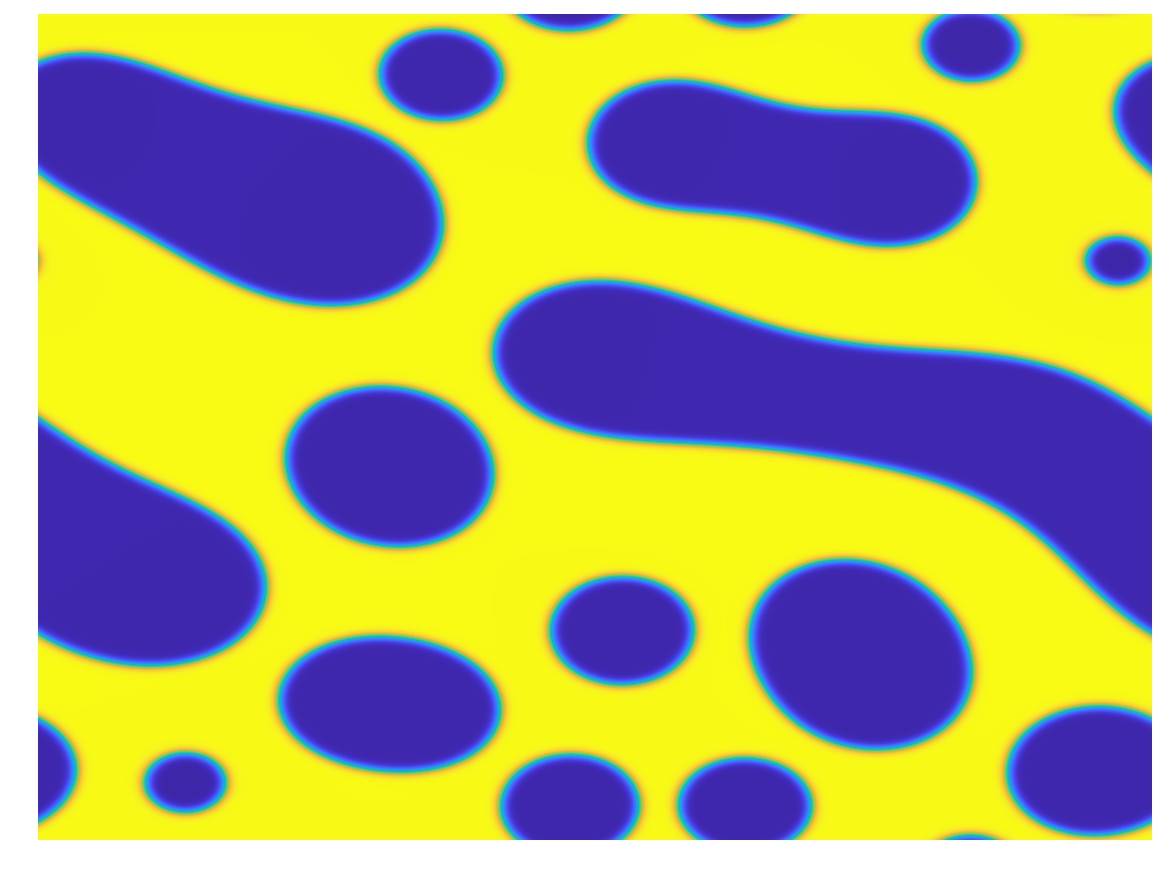}
			\includegraphics[height=0.48\textwidth,width=0.48\textwidth]{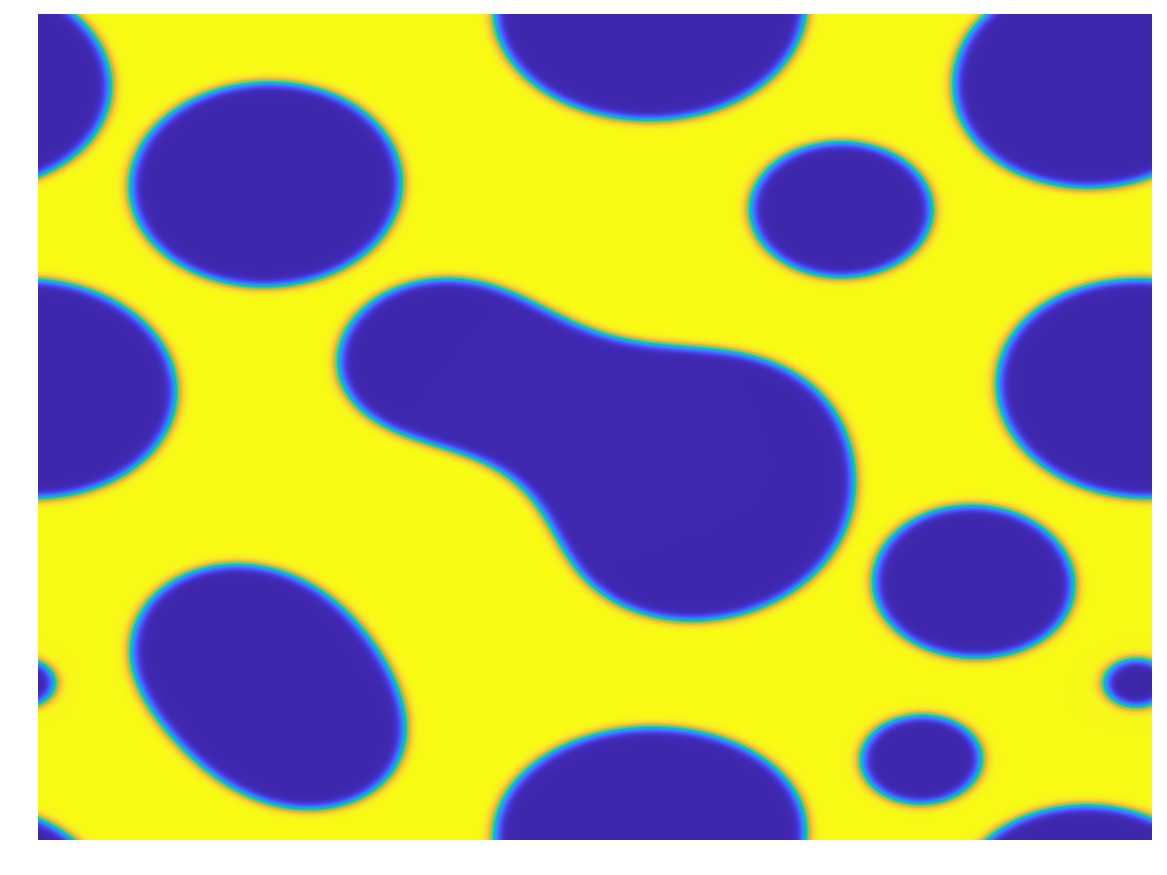}
			\caption*{$t=0.2, 0.5$}
		\end{subfigure}
		\begin{subfigure}{0.48\textwidth}
			\includegraphics[height=0.48\textwidth,width=0.48\textwidth]{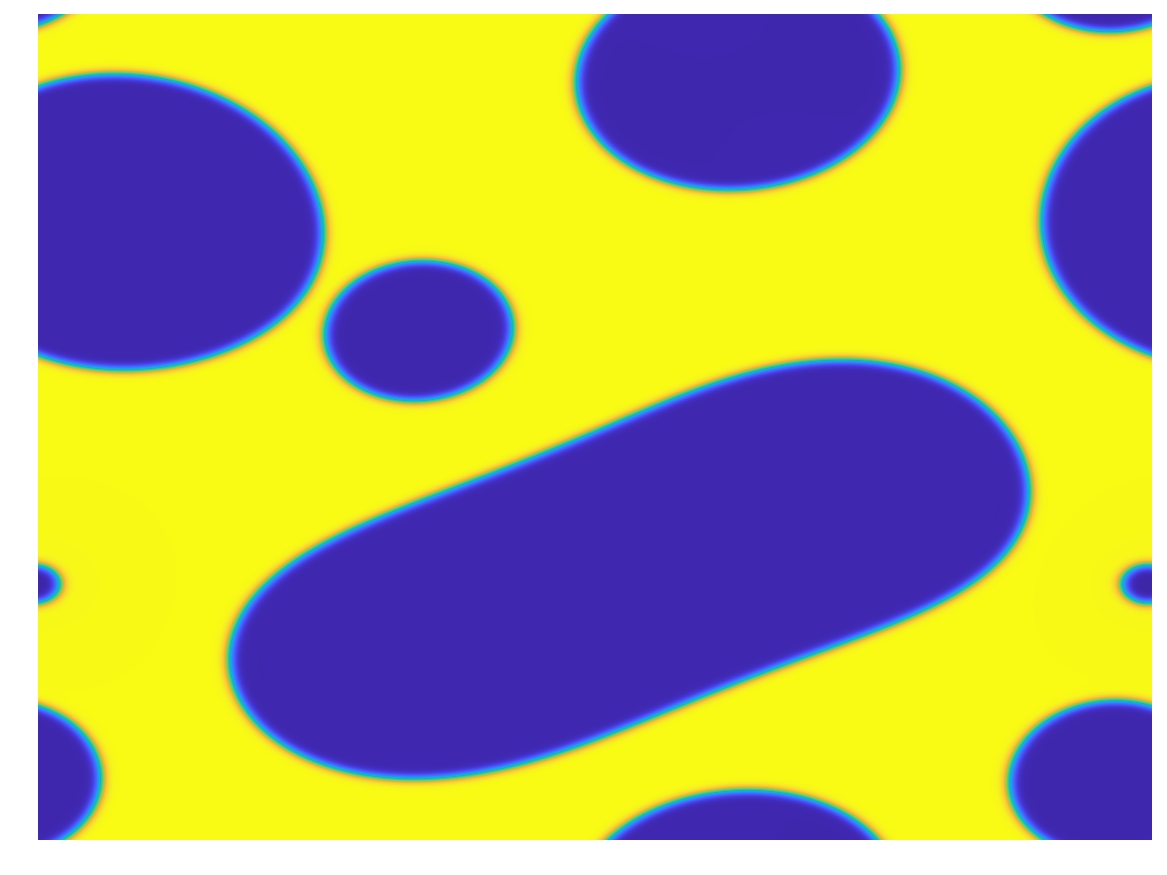}
			\includegraphics[height=0.48\textwidth,width=0.48\textwidth]{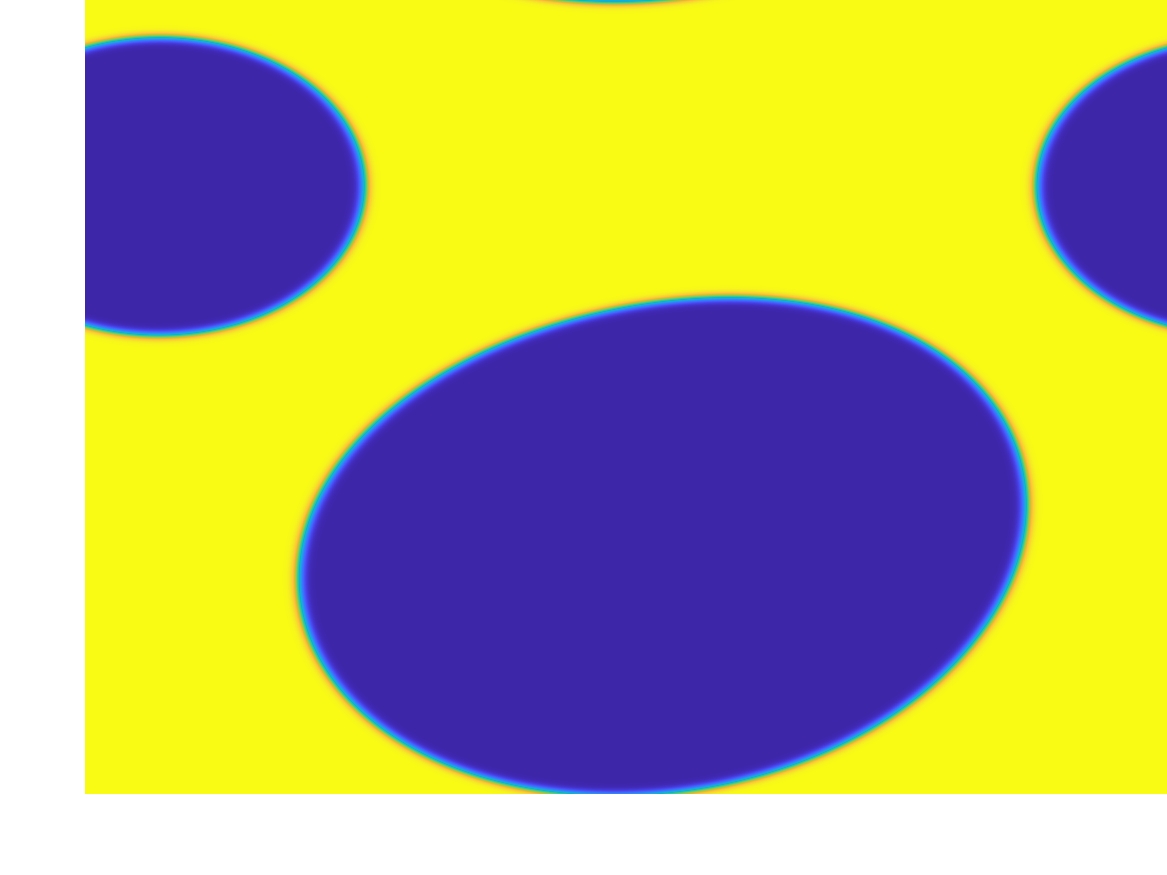}
			\caption*{$t=1, 3$}
		\end{subfigure}
		\begin{subfigure}{0.48\textwidth}
			\includegraphics[height=0.48\textwidth,width=0.48\textwidth]{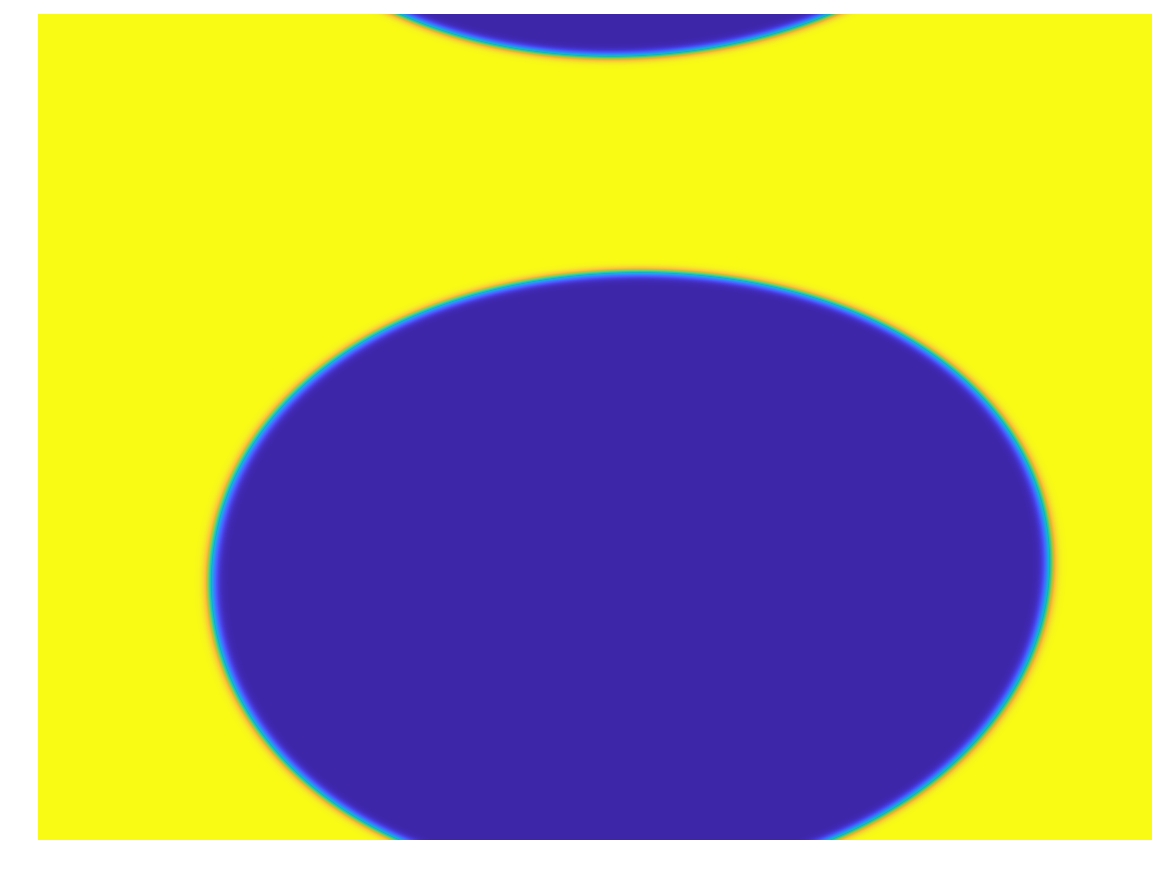}
			\includegraphics[height=0.48\textwidth,width=0.48\textwidth]{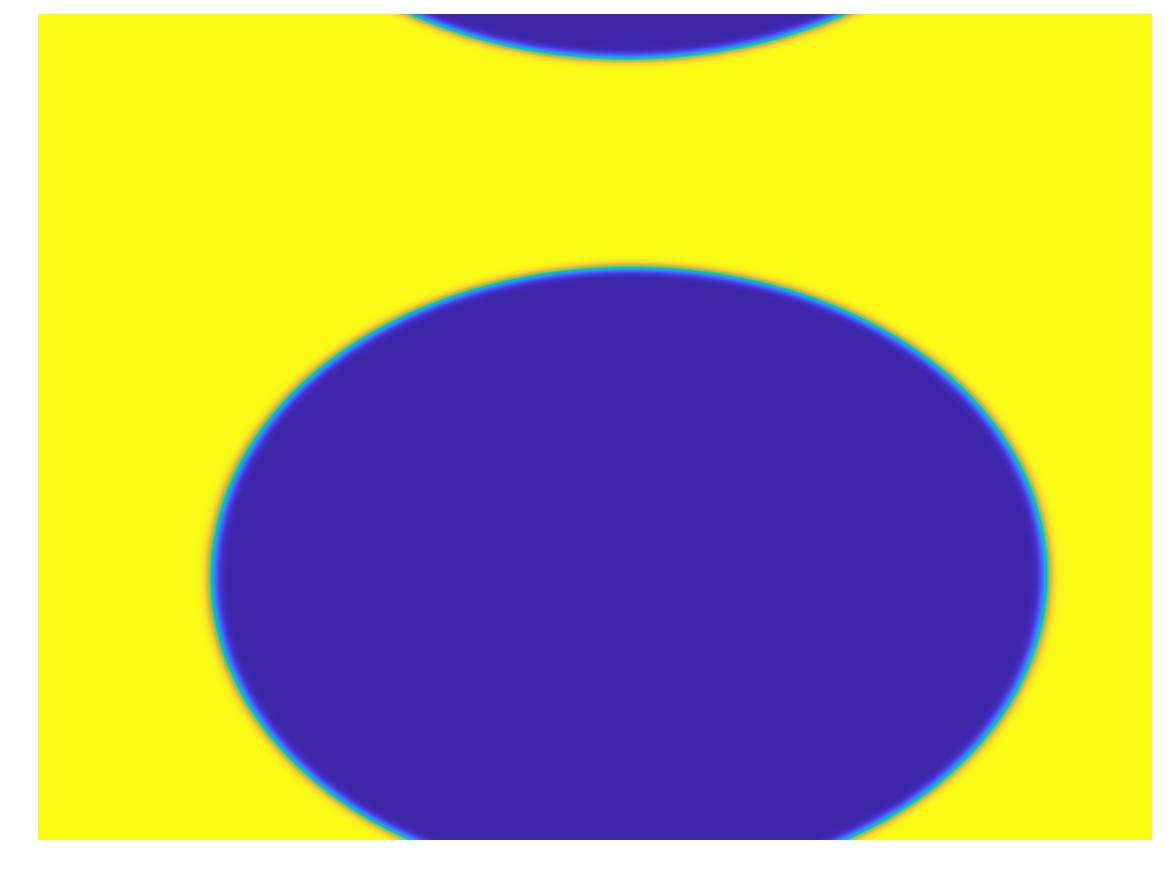}
			\caption*{$t=7, 15$}
		\end{subfigure}
\caption{(Color online.) Snapshots of the phase variable at the indicated time instants over the domain $\Omega = [0,1]^2$, $\varepsilon = 0.005$, $\theta_0=3$, with a constant mobility ${\cal M} \equiv 1$. }
		\label{fig3}
	\end{center}
\end{figure}

To investigate whether the strict separation property is satisfied for the proposed numerical solver, we display the maximum and minimum values of the phase variable at the associated time sequence in Table 1. A safe distance, with an order of $O (10^{-1})$, between the numerical solution and the singular limit values of $\pm 1$, is clearly observed in the simulation. This numerical result also confirms the strict separation estimate established in the theoretical analysis~\cite{abels07, debussche95}. In fact, the spatially uniform equilibria solution turns out to be $\phi_* \equiv 0.8586$ %0.858559636645475$ 
for such an expansive parameter value of $\theta_0=3$. The maximum and minimum values in Table 1 reveal that, the numerical solution in principle stays within the interval $[- \phi_*, \phi_*]$, %the values of $\pm \phi_*$ of the double well minima in the 2-D simulation, 
while a minor deviation of order $ O(10^{-2})$ is observed from time to time. Such a minor deviation comes from the fact that, the Cahn-Hilliard equation does not preserve the bound of $[- \phi_*, \phi_*]$, in comparison with the Allen-Cahn equation, in which the maximum principle could be rigorously justified. Instead, the 2-D Cahn-Hilliard equation preserves a separation property, $-1 + \epsilon_0 \le \phi \le 1 - \epsilon_0$, in which $\epsilon_0$ depends on $\varepsilon$ and $\theta_0$, while $1 - \epsilon_0 \ne \phi_*$.  

 \begin{table}[!htb] 
\begin{center}
\begin{tabular}{ccc}
%\hline &&\multicolumn{2}{c}{FAS}\\
%\hline &&\multicolumn{4}{c}{$$}&\multicolumn{4}{c}{MG}\\
\hline  {\bf Time instants} 
                           & {\bf the maximum value}   
                           & {\bf the minimum value}         \\
\hline $t_1 =0.05$  & 0.8643153 & -0.8744858  \\
  $t_2 =0.1$   & 0.8602455  &  -0.8700498   \\
   $t_3 =0.2$  & 0.8589113  &  -0.8740629  \\
    $t_4 =0.5$  & 0.8598682  &  -0.8645398  \\  
    $t_5 =1$  & 0.8577464  &  -0.8755585  \\  
   $t_6 = 3$  & 0.8571105 &  -0.8611945  \\    
   $t_7 =7$  & 0.857263  &  -0.8600106  \\  
   $t_8 = 15$  & 0.8571818 & -0.8599258 \\    
\hline
\end{tabular}
\caption{The maximum and minimum values of of the phase variable at the indicated time instants over the domain $\Omega = [0,1]^2$, $\varepsilon = 0.005$, $\theta_0=3$.}
\end{center}
 \label{table:1} 
\end{table}

Furthermore, the long time characteristics of the solution, such as the energy decay rate, are of great scientific interest. The $t^{-1/3}$ energy decay scaling law has been reported for the Cahn-Hilliard flow with a polynomial approximation energy potential, at both the theoretical and numerical levels~\cite{cheng2019a, cheng16a, kohn03}. Meanwhile, such a theoretical analysis has not been available for the energy potential with Flory-Huggins logarithmic energy potential. A numerical experiment for a $t^{-b^*}$ (with $b^*$ close to $-\frac13$) scaling law was reported in a recent work~\cite{chen22a}, based on a second order accurate scheme for the Flory-Huggins-Cahn-Hilliard flow. In this article, we provide numerical evidence of this scaling law. Figure~\ref{fig:energy evolution} presents the log-log plot for the energy versus time, based on the PSD iteration solver for the numerical scheme~\eqref{eqn:scheme}. The detailed scaling ``exponent" is obtained using least squares fits of the computed data up to time $t=100$.  A clear observation of the $a_e t^{b_e}$ scaling law can be made,  with $a_e = 0.01933$, $b_e=-0.3271$. 
%It is amazing to obtain an energy dissipation scaling index for the Flory-Huggins Cahn-Hilliard flow in the long time numerical simulation, which is close to the  $t^{-1/3}$ scaling observed in the polynomial approximation model.

	\begin{figure}
	\begin{center}
\includegraphics[width=3.0in]{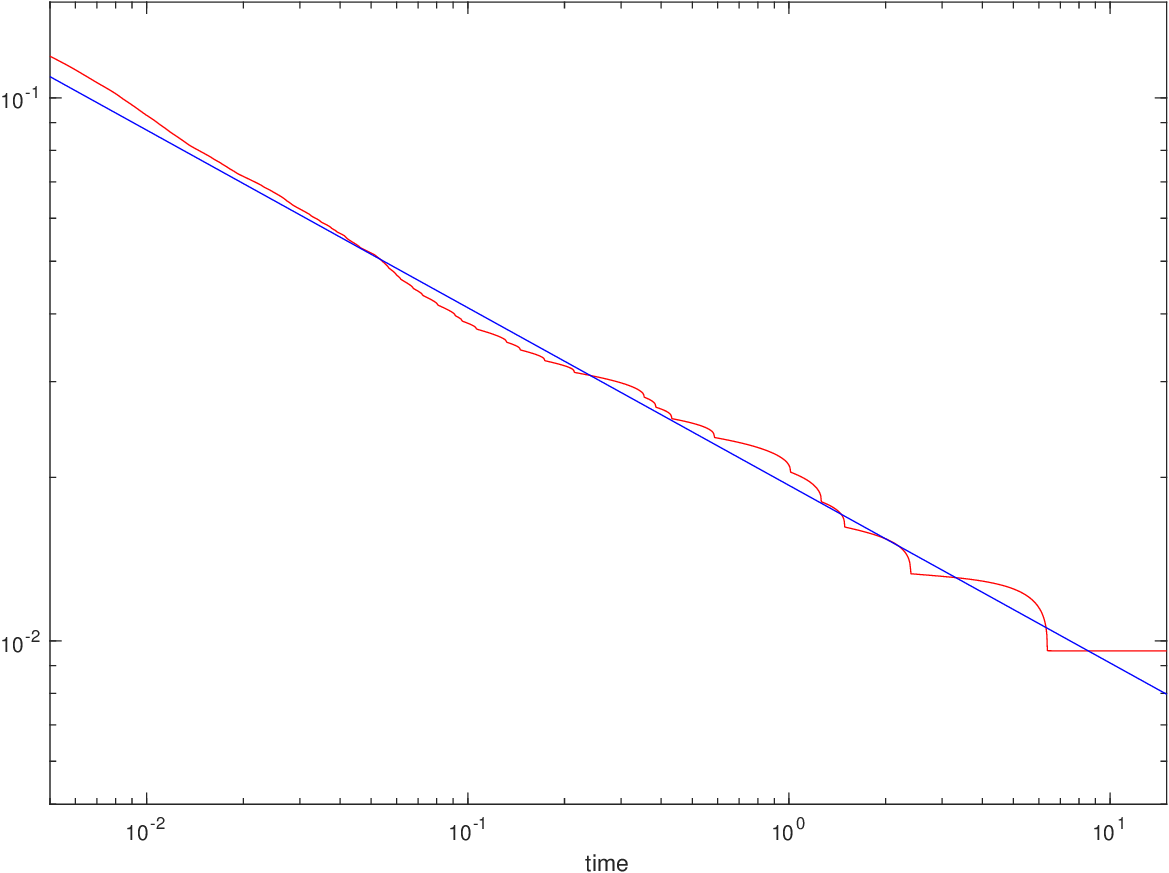}
	\end{center}
\caption{Log-log plot of the temporal evolution of the discrete energy for $\varepsilon=0.005$, $\theta_0=3$, with a constant mobility ${\cal M} \equiv 1$.  The energy decreases similar to $a_e t^{b_e}$ until saturation. The red line represents the energy plot obtained by the simulations, while the straight blue line is obtained by least squares approximations to the energy data.  The least squares fit is only taken for the linear part of the calculated data, and only up to $t=100$. The fitted line has the form $a_e t^{b_e}$, with $a_e = 0.01933$, $b_e=-0.3271$.}
 	\label{fig:energy evolution}
 	\end{figure}

\section{Concluding remarks}  \label{sec:conclusion}

In this article, the preconditioned steepest descent (PSD) iteration solver is considered to implement a finite difference  numerical scheme for the Cahn-Hilliard equation with Flory-Huggins energy potential. A convex-concave decomposition is applied to the energy functional, and the convex splitting numerical approximation to the chemical potential:  implicit treatment for the singular logarithmic term and the surface diffusion term, combined with an explicit update for the expansive concave term. The positivity-preserving analysis, unconditional energy stability, and the optimal rate error estimate have been theoretically derived in a recent work. In terms of the numerical implementation of this nonlinear and singular numerical scheme, we propose a preconditioned steepest descent iteration solver in the computation, based on the fact that the implicit parts of the numerical scheme are associated with a strictly convex energy. This iteration solver consists of a computation of the search direction (involved with a Poisson-like equation), and a one-parameter optimization over the search direction. At a theoretical level, a geometric convergence rate is proved for the PSD iteration, and the positivity-preserving property is theoretically established at each iteration stage in the process. Moreover, a uniform distance estimate between the numerical solution and the singular limit values of  $\pm 1$ for the phase variable has played an essential role in the theoretical analysis. Such an iteration convergence analysis and positivity-preserving analysis is a first for a phase field model with a singular energy potential. A few numerical examples are presented to demonstrate the robustness and efficiency of the PSD solver.

	\section*{Acknowledgements}
This work is supported in part by the National Science Foundation (USA) grants NSF DMS-2012269 and DMS-2309548 (C.~Wang), NSF DMS-2012634 and DMS-2309547 (S.M.~Wise), and NSF DMS-2110768 (A.~Diegel).

\bibliographystyle{plain}
	\bibliography{revision2}

	\end{document}